\documentclass[12pt]{amsart}
\usepackage{amsmath}
\usepackage{amssymb}
\usepackage{amsfonts}
\usepackage{amscd}
\usepackage{thmdefs}
\usepackage[pagebackref,hypertexnames=false, colorlinks, citecolor=red, linkcolor=red]{hyperref}
\usepackage[backrefs]{amsrefs}
\usepackage[color=green!15,bordercolor=red,linecolor=red!30]{todonotes}
\usepackage{ulem}

\theoremstyle{definition}
\theoremstyle{remark}
\theoremstyle{question}
\numberwithin{equation}{section}

\newtheorem{question}{Question}[section]

\begin{document}
\title{The Linear Bound for the Natural Weighted Resolution of the Haar Shift}
\author[S. Pott]{Sandra Pott}
\thanks{}
\address{S. Pott, Centre for Mathematical Sciences, University of Lund,
Lund, Sweden}
\email{sandra@maths.lth.se}
\author[M.C. Reguera]{Maria Carmen Reguera$^1$}
\thanks{1. Research supported by grants 2009SGR-000420 (Generalitat de
Catalunya) and MTM-2010-16232 (Spain)}
\address{M. C. Reguera, Department of Mathematics, Universitat Aut\`onoma de
Barcelona, Barcelona, Spain}
\email{mreguera@mat.uab.cat}
\author[E. T. Sawyer]{Eric T. Sawyer$^2$}
\thanks{2. Research supported in part by a NSERC Grant.}
\address{E. T. Sawyer, Department of Mathematics, McMaster University,
Hamilton, Canada}
\email{sawyer@mcmaster.ca}
\author[B. D. Wick]{Brett D. Wick$^3$}
\address{Brett D. Wick, School of Mathematics\\
Georgia Institute of Technology\\
686 Cherry Street\\
Atlanta, GA USA 30332-0160}
\email{wick@math.gatech.edu}
\urladdr{www.math.gatech.edu/~wick}
\thanks{3. Research supported in part by National Science Foundation DMS
grants \# 1001098 and \# 955432.}
\thanks{The authors would like to thank the Banff International Research
Station for the Banff--PIMS Research in Teams support for the project: The
Sarason Conjecture and the Composition of Paraproducts.}

\begin{abstract}
The Hilbert transform has a linear bound in the $A_{2}$ characteristic on
weighted $L^{2}$, 
\begin{equation*}
\left\Vert H\right\Vert _{L^{2}(w)\rightarrow L^{2}(w)}\lesssim \left[ w%
\right] _{A_{2}},
\end{equation*}%
and we extend this linear bound to the nine constituent operators in the
natural weighted resolution of the conjugation $M_{w^{\frac{1}{2}}}\mathcal{S%
}M_{w^{-\frac{1}{2}}}$ induced by the canonical decomposition of a
multiplier into paraproducts:%
\begin{equation*}
M_{f}=P_{f}^{-}+P_{f}^{0}+P_{f}^{+}.
\end{equation*}%
The main tools used are composition of paraproducts, a \textquotedblleft product
formula\textquotedblright\ for Haar coefficients, the Carleson Embedding
Theorem, and the linear bound for the square function.
\end{abstract}

\maketitle
\tableofcontents

\section{Introduction}

Let $L^{2}\equiv L^{2}\left( \mathbb{R}\right) $ denote the space of square
integrable functions over $\mathbb{R}$. For a weight $w$, i.e., a positive
locally integrable function on $\mathbb{R}$, we set $L^{2}(w)\equiv L^{2}(%
\mathbb{R};w)$. In particular, we will be interested in $A_{2}$ weights,
which are defined by finiteness of their $A_{2}$ characteristic, 
\begin{equation*}
\left[ w\right] _{A_{2}}\equiv \sup_{I}\left\langle w\right\rangle
_{I}\left\langle w^{-1}\right\rangle _{I},
\end{equation*}%
where $\left\langle w\right\rangle _{I}$ denotes the average of $w$ over the
interval $I$.

An operator $T$ is bounded on $L^{2}(w)$ if and only if $M_{w^{\frac{1}{2}%
}}TM_{w^{-\frac{1}{2}}}$ - the conjugation of $T$ by the multiplication
operator $M_{w^{\pm \frac{1}{2}}}$ - is bounded on $L^{2}$. Moreover, the
operator norms are equal:%
\begin{equation*}
\left\Vert T\right\Vert _{L^{2}(w)\rightarrow L^{2}(w)}=\left\Vert M_{w^{%
\frac{1}{2}}}TM_{w^{-\frac{1}{2}}}\right\Vert _{L^{2}\rightarrow L^{2}}.
\end{equation*}%
In the case that $T$ is a dyadic operator adapted to a dyadic grid $\mathcal{%
D}$, it is natural to study weighted norm properties of $T$ by decomposing
the multiplication operators $M_{w^{\pm \frac{1}{2}}}$ into their canonical
paraproduct decomposition relative to the grid $\mathcal{D}$ and Haar basis $%
\left\{ h_{I}^{0}\right\} _{I\in \mathcal{D}}$, i.e.%
\begin{eqnarray*}
M_{w^{\pm \frac{1}{2}}}f &=&\mathsf{P}_{\widehat{w^{\pm \frac{1}{2}}}%
}^{(0,1)}+\mathsf{P}_{\widehat{w^{\pm \frac{1}{2}}}}^{(1,0)}f+\mathsf{P}%
_{\langle w^{\pm \frac{1}{2}}\rangle }^{(0,0)}f \\
&\equiv &\sum_{I\in \mathcal{D}}\left\langle w^{\pm \frac{1}{2}%
},h_{I}^{0}\right\rangle _{L^{2}}\left\langle f,h_{I}^{1}\right\rangle
_{L^{2}}h_{I}^{0} \\
&&+\sum_{I\in \mathcal{D}}\left\langle w^{\pm \frac{1}{2}},h_{I}^{0}\right%
\rangle _{L^{2}}\left\langle f,h_{I}^{0}\right\rangle _{L^{2}}h_{I}^{1} \\
&&+\sum_{I\in \mathcal{D}}\left\langle w^{\pm \frac{1}{2}},h_{I}^{1}\right%
\rangle _{L^{2}}\left\langle f,h_{I}^{0}\right\rangle _{L^{2}}h_{I}^{0},
\end{eqnarray*}%
(above $h_{I}^{1}$ is the averaging function) and then decomposing $M_{w^{%
\frac{1}{2}}}TM_{w^{-\frac{1}{2}}}$ into the nine canonical individual
paraproduct composition operators:%
\begin{eqnarray}
M_{w^{\frac{1}{2}}}TM_{w^{-\frac{1}{2}}} &=&\left( \mathsf{P}_{\widehat{w^{%
\frac{1}{2}}}}^{(0,1)}+\mathsf{P}_{\widehat{w^{\frac{1}{2}}}}^{(1,0)}+%
\mathsf{P}_{\langle w^{\frac{1}{2}}\rangle }^{(0,0)}\right) T\left( \mathsf{P%
}_{\widehat{w^{-\frac{1}{2}}}}^{(0,1)}+\mathsf{P}_{\widehat{w^{-\frac{1}{2}}}%
}^{(1,0)}+\mathsf{P}_{\langle w^{-\frac{1}{2}}\rangle }^{(0,0)}\right)
\label{canoncialpara} \\
&\equiv &Q_{T,w}^{\left( 0,1\right) ,\left( 0,1\right) }+Q_{T,w}^{\left(
0,1\right) ,\left( 1,0\right) }+Q_{T,w}^{\left( 0,1\right) ,\left(
0,0\right) }  \notag \\
&&+Q_{T,w}^{\left( 1,0\right) ,\left( 0,1\right) }+Q_{T,w}^{\left(
1,0\right) ,\left( 1,0\right) }+Q_{T,w}^{\left( 1,0\right) ,\left(
0,0\right) }  \notag \\
&&+Q_{T,w}^{\left( 0,0\right) ,\left( 0,1\right) }+Q_{T,w}^{\left(
0,0\right) ,\left( 1,0\right) }+Q_{T,w}^{\left( 0,0\right) ,\left(
0,0\right) },  \notag
\end{eqnarray}%
where $Q_{T,w}^{\left( \varepsilon _{1},\varepsilon _{2}\right) ,\left(
\varepsilon _{3},\varepsilon _{4}\right) }\equiv \mathsf{P}_{\widehat{w^{%
\frac{1}{2}}}}^{(\varepsilon _{1},\varepsilon _{2})}T\mathsf{P}_{\widehat{%
w^{-\frac{1}{2}}}}^{(\varepsilon _{3},\varepsilon _{4})}$. We refer to this
decomposition of the conjugation $M_{w^{\frac{1}{2}}}TM_{w^{-\frac{1}{2}}}$
into nine operators as the \emph{natural weighted resolution} of the
operator $T$.

The operators we will be interested in have the property that $\left\Vert
T\right\Vert _{L^{2}(w)\rightarrow L^{2}(w)}\lesssim \left[ w\right]
_{A_{2}} $. And we will be interested in demonstrating that the operator norms of the $Q_{T,w}^{\left( \varepsilon
_{1},\varepsilon _{2}\right) ,\left( \varepsilon _{3},\varepsilon
_{4}\right) }$ are linear in the $A_{2}$ characteristic:%
\begin{equation*}
\left\Vert Q_{T,w}^{\left( \varepsilon _{1},\varepsilon _{2}\right) ,\left(
\varepsilon _{3},\varepsilon _{4}\right) }\right\Vert _{L^{2}\rightarrow
L^{2}}\lesssim \left[ w\right] _{A_{2}}.
\end{equation*}%
%
%
%
%
%
%
%

We wish to apply this general idea to the Haar shift operator.  This is the following dyadic model operator 
\begin{equation*}
M_{w^{\frac{1}{2}}}\mathcal{S}M_{w^{-\frac{1}{2}}}:L^{2} \rightarrow L^{2}
\end{equation*}%
where $\mathcal{S}$ is a shift operator defined on the Haar basis by $%
h_I\mapsto h_{I_-}-h_{I_+}$. Because of linearity, it suffices to consider
just ``half'' of the shift operator $\mathcal{S}$ defined on the Haar basis
by the operator $\mathcal{S} h_I\equiv h_{I_-}$.  Petermichl proved the following result:

\begin{theorem}[Petermichl, \protect\cite{Petermichl}]
\label{linearA_2} Let $w\in A_{2}$. Then 
\begin{equation*}
\left\Vert \mathcal{S}\right\Vert _{L^{2}(w)\rightarrow L^{2}(w)}\lesssim \left[ w%
\right] _{A_{2}}.
\end{equation*}%
\end{theorem}
This result can be used with an averaging technique, \cite{Petermichl2}, to show that the same estimate persists for the Hilbert transform on $L^2(w)$.  The idea of decomposing a general Calder\'on--Zygmund operator into shifts was further refined by Hyt\"onen in \cites{H, H2} (see also \cites{V,MR2657437, MR1964822}) and was used in his solution of the $A_2$-Conjecture for all Calder\'on--Zygmund operators.  For the simplest proof of the $A_2$-Conjecture the interested reader should consult \cite{L}.

In this paper we will implement the strategy outlined above to extend this result to the natural weighted resolution of the Haar shift transform and obtain the following result.

\begin{theorem}
\label{linearA_2_pararesolution} Let $w\in A_2$ and $\mathcal{S}$ the Haar
shift on $L^2$. Then for the resolution of $\mathcal{S}$ into its canonical
paraproducts as given in \eqref{canoncialpara} we have that each term can be
controlled by a linear power of $\left[ w\right]_{A_2}$. In particular, 
\begin{equation*}
\left\Vert Q_{\mathcal{S},w}^{\left( \varepsilon _{1},\varepsilon
_{2}\right) ,\left( \varepsilon _{3},\varepsilon _{4}\right) }\right\Vert
_{L^{2}\rightarrow L^{2}}\lesssim \left[ w\right] _{A_{2}}.
\end{equation*}
\end{theorem}

If one could control each term appearing in Theorem \ref{linearA_2_pararesolution} independent of the Theorem \ref{linearA_2}, this would imply Petermichl's Theorem.  But, for one of the terms, we unfortunately need to resort to the estimate in Theorem \ref{linearA_2}, and will point out an interesting question that we are unable to resolve as of this writing.  However, our point is to demonstrate that the paraproduct operators arising in the canonical resolution of the Haar shift $\mathcal{S}$ are also bounded, and linearly in the $A_{2}$ characteristic.

Finally, we mention that the operators $Q_{\mathcal{I},w}^{\left(\varepsilon _{1},\varepsilon _{2}\right) ,\left( \varepsilon_{3},\varepsilon _{4}\right) }$ in the resolution of the identity $\mathcal{I}=M_{w^{\frac{1}{2}}}\mathcal{I}M_{w^{-\frac{1}{2}}}$ in (\ref{canoncialpara}), are all bounded on $L^{2}$ \emph{if and only if} $w\in A_{2}$, and if so, the operator norm is linear in the $A_{2}$ characteristic. The proof of this fact is easy using the techniques in the proof of Theorem \ref{linearA_2_pararesolution}, and amounts to little more than repeating the arguments without the shift involved.  Key to this proof working is that the identity is a purely local operator, and this points to the difficulty of Haar shift and its non-local nature.

Throughout this paper $\equiv $ means equal by definition, while $A\lesssim
B $ means that there exists an absolute constant $C$ such that $A\leq CB$.

\section{Notation and Preliminary Estimates}

Before proceeding with the proof of Theorem \ref{linearA_2_pararesolution}, we collect a few elementary
observations and necessary notation that will be used frequently throughout
the remainder of the paper.

To define our paraproducts, let $\mathcal{D}$ denote the usual dyadic grid
of intervals on the real line. Define the Haar function $h_{I}^{0}$ and
averaging function $h_{I}^{1}$ by 
\begin{equation*}
h_{I}^{0}\equiv h_{I}\equiv \frac{1}{\sqrt{\left\vert I\right\vert }}\left( 
\mathbf{1}_{I_{-}}-\mathbf{1}_{I_{+}}\right) \text{ and }h_{I}^{1}\equiv 
\frac{1}{\left\vert I\right\vert }\mathbf{1}_{I}\ ,\ \ \ \ \ I\in \mathcal{D}.
\end{equation*}
The paraproduct operators considered in this paper are the following dyadic
operators.

\begin{definition}
\label{Paraproduct_Def} Given a symbol $b=\left\{ b_{I}\right\} _{I\in 
\mathcal{D}}$ and a pair $\left( \alpha ,\beta \right) \in \left\{
0,1\right\} \times \left\{ 0,1\right\} $, define the \emph{dyadic paraproduct%
} acting on a function $f$ by 
\begin{equation*}
\mathsf{P}_{b}^{\left( \alpha ,\beta \right) }f\equiv \sum_{I\in \mathcal{D}%
}b_{I}\left\langle f,h_{I}^{\beta }\right\rangle_{L^2} h_{I}^{\alpha },
\end{equation*}%
where $h_{I}^{0}$ is the Haar function associated with $I$, and $h_{I}^{1}$
is the average function associated with $I$. The index $\left( \alpha ,\beta
\right) $ is referred to as the \emph{type} of $\mathsf{P}_{b}^{\left(
\alpha ,\beta \right) }$.
\end{definition}

For a function $b$ and $I\in \mathcal{D}$ we let 
\begin{eqnarray*}
\widehat{b}(I) &\equiv &\left\langle b,h_{I}^{0}\right\rangle _{L^{2}} \\
\left\langle b\right\rangle _{I} &\equiv &\left\langle
b,h_{I}^{1}\right\rangle _{L^{2}},
\end{eqnarray*}%
denote the corresponding sequences indexed by the dyadic intervals $I\in 
\mathcal{D}$. With this notation, the canonical paraproduct decomposition of
the pointwise multiplier operator $M_{b}$ is given by 
\begin{equation*}
M_{b}=P_{\widehat{b}}^{(0,1)}+P_{\widehat{b}}^{(1,0)}+P_{\langle b\rangle
}^{(0,0)}.
\end{equation*}

At points in the argument below we will have to resort to the use of
disbalanced Haar functions. To do so, we introduce some additional notation.
Given a weight $\sigma $ on $\mathbb{R}$ we set 
\begin{equation}
C_{K}(\sigma )\equiv \sqrt{\frac{\left\langle \sigma \right\rangle
_{K_{+}}\left\langle \sigma \right\rangle _{K_{-}}}{\left\langle \sigma
\right\rangle _{K}}}\quad \text{ and }\quad D_{K}(\sigma )\equiv \frac{%
\widehat{\sigma }(K)}{\left\langle \sigma \right\rangle _{K}}.
\label{disbalanced_defs1}
\end{equation}%
Then we have 
\begin{equation}
h_{K}=C_{K}(\sigma )h_{K}^{\sigma }+D_{K}(\sigma )h_{K}^{1}
\label{disbalanced_defs2}
\end{equation}%
where $\left\{ h_{K}^{\sigma }\right\} _{K\in \mathcal{D}}$ is the $%
L^{2}(\sigma )$ normalized Haar basis. For an interval $K\in \mathcal{D}$ we
let 
\begin{equation*}
\mathbb{E}_{K}^{\sigma }(f)\equiv \frac{1}{\sigma (K)}\int_{K}f\sigma dx
\end{equation*}%
where $\sigma (K)=\int_{K}\sigma dx$. When $\sigma $ is Lebesgue measure we
simply write $\left\langle f\right\rangle_K=\mathbb{E}_{K}(f)\equiv \frac{1}{%
\left\vert K\right\vert }\int_{K}fdx$.


For two functions $f$ and $g$ we have the following \textquotedblleft
product formula\textquotedblright\ that appears in higher dimensions in \cite%
{SSU} 
\begin{equation*}
\widehat{fg}\left( I\right) =\sum_{J\subsetneq I}\widehat{f}\left( J\right) 
\widehat{g}\left( J\right) \frac{\delta \left( J,I\right) }{\sqrt{\left\vert
I\right\vert }}+\widehat{f}\left( I\right) \left\langle g\right\rangle _{I}+%
\widehat{g}\left( I\right) \left\langle f\right\rangle _{I}\ \quad \forall
I\in \mathcal{D},
\end{equation*}%
where $\delta (J,I)=1$ if $J\subset I_{-}$ and $\delta (J,I)=-1$ if $%
J\subset I_{+}$.

For a sequence $a=\{a_I\}_{I\in\mathcal{D}}$ define. 
\begin{eqnarray*}
\left\Vert a\right\Vert _{\ell ^{\infty }} &\equiv &\sup_{I\in \mathcal{D}%
}\left\vert a_{I}\right\vert ; \\
\left\Vert a\right\Vert _{CM} &\equiv &\sqrt{\sup_{I\in \mathcal{D}}\frac{1}{%
\left\vert I\right\vert }\sum_{J\subset I}\left\vert a_{J}\right\vert^{2}}\ .
\end{eqnarray*}

The following estimates are well-known.

\begin{lemma}
\label{ClassicalCharacterization} We have the characterizations: 
\begin{eqnarray}
\left\Vert \mathsf{P}_{a}^{\left( 0,0\right) }\right\Vert _{L^{2}\rightarrow
L^{2}} &=&\left\Vert a\right\Vert _{\ell ^{\infty }};  \label{simple} \\
\left\Vert \mathsf{P}_{a}^{\left( 0,1\right) }\right\Vert _{L^{2}\rightarrow
L^{2}} &=&\left\Vert \mathsf{P}_{a}^{\left( 1,0\right) }\right\Vert
_{L^{2}\rightarrow L^{2}}\approx \left\Vert a\right\Vert _{CM}\ .  \label{CM}
\end{eqnarray}
\end{lemma}

The Carleson Embedding Theorem is a fundamental tool in this paper and will
be used frequently. It is the following:

\begin{theorem}[Carleson Embedding Theorem]
Let $v\geq 0$, and let $\{\alpha _{I}\}_{I\in \mathcal{D}}$ be positive
constants. The following two statements are equivalent: 
\begin{eqnarray*}
\sum_{I\in \mathcal{D}}\alpha _{I}\,\mathbb{E}_{I}^{v}\left( f\right) ^{2}
&\leq &4C\left\Vert f\right\Vert _{L^{2}(v)}^{2}\quad \forall f\in L^{2}(v);
\\
\sup_{I\in \mathcal{D}}v(I)^{-1}\sum_{J\subset I}\alpha _{J}\left\langle
v\right\rangle _{J}^{2} &\leq &C.
\end{eqnarray*}
\end{theorem}

As a simple application of the Carleson Embedding Theorem, one can prove:

\begin{lemma}
\label{ClassicalCharacterization2} Let $a=\{a_I\}_{I\in\mathcal{D}}$ be a
sequence of non-negative numbers. Then 
\begin{equation}  \label{CET}
\left\Vert \mathsf{P}_{a}^{\left( 1,1\right) }\right\Vert _{L^{2}\rightarrow
L^{2}}\lesssim \left\|a^{\frac{1}{2}}\right\|_{CM}^2.
\end{equation}
\end{lemma}

Recall that the dyadic square function is given by 
\begin{equation*}
S\phi (x)\equiv \sqrt{\sum_{I\in \mathcal{D}}\left\vert \widehat{\phi }%
(I)\right\vert^{2}\,h_{I}^{1}(x)},
\end{equation*}%
and that for any weight $v\geq 0$ we have%
\begin{equation*}
\left\Vert S\phi \right\Vert _{L^{2}\left( v\right) }^{2}=\sum_{I\in 
\mathcal{D}}\left\vert \widehat{\phi }(I)\right\vert^{2}\left\langle v\right\rangle _{I}\ .
\end{equation*}%
Since $\{h_{I}\}_{I\in \mathcal{D}}$ is an orthonormal basis for $L^{2}(%
\mathbb{R})$, it is trivial that $\left\Vert Sf\right\Vert
_{L^{2}}=\left\Vert f\right\Vert _{L^{2}}$. However, in \cite{PetermichlPott}
it was shown for $w\in A_{2}$ that 
\begin{equation*}
\left\Vert Sf\right\Vert _{L^{2}\left( w\right) }^{2}\lesssim \left[ w\right]
_{A_{2}}^{2}\left\Vert f\right\Vert _{L^{2}(w)}^{2}.
\end{equation*}%
Applying this inequality to $f=w^{-\frac{1}{2}}\mathbf{1}_{I}$ for $I\in 
\mathcal{D}$, along with some obvious estimates, yields the following: 
\begin{equation}
\sum_{J\subset I}\widehat{w^{-\frac{1}{2}}}(J)^{2}\left\langle
w\right\rangle _{J}\lesssim \left[ w\right] _{A_{2}}^{2}|I|\quad \forall
I\in \mathcal{D}.  \label{SquareFunctionEst}
\end{equation}%
A trivial consequence of \eqref{SquareFunctionEst} is, 
\begin{equation}
\sum_{J\subset I}\widehat{w^{-\frac{1}{2}}}(J)^{2}\left\langle w^{\frac{1}{2}%
}\right\rangle _{J}^{2}\lesssim \left[ w\right] _{A_{2}}^{2}|I|\quad \forall
I\in \mathcal{D},  \label{SquareFunctionEst2}
\end{equation}%
since $\left\langle w^{\frac{1}{2}}\right\rangle _{J}^{2}\leq \left\langle
w\right\rangle _{J}$. Because of the symmetry of the $A_{2}$ condition, we
also have these estimates with the roles of $w$ and $w^{-1}$ interchanged.
All of these estimates play a fundamental role at various times when
applying the Carleson Embedding Theorem.

We also need a modified version of estimate \eqref{SquareFunctionEst}, but
which incorporates a shift in the indices. Define the modified square
function $S_{\pi }$ by, 
\begin{equation*}
S_{\pi }\phi (x)\equiv \sqrt{\sum\limits_{I\in \mathcal{D}}\left\vert 
\widehat{\phi }\left( I\right) \right\vert^{2}\,\frac{1}{\left\vert
I\right\vert }\mathbf{1}_{\pi I}(x)}
\end{equation*}%
where $\pi I$ is the dyadic parent of the dyadic interval $I$.
Unfortunately, $S_{\pi }f$ is not pointwise bounded by $Sf$ ($S_{\pi
}h_{K}\equiv 1$ on $\pi K$ but $Sh_{K}$ vanishes outside $K$), yet we show
in the theorem below that $S_{\pi }$ has a linear bound in the
characteristic on weighted spaces.

\begin{theorem}
\label{ModifiedSquare} For any $\phi\in L^2(w)$ we have 
\begin{equation*}
\left\Vert S_{\pi }\phi \right\Vert _{L^{2}\left( w\right) }\lesssim \left[ w%
\right] _{A_{2}}\left\Vert \phi\right\Vert _{L^{2}\left( w\right) }.
\end{equation*}
\end{theorem}

\begin{proof}
The proof of this theorem uses the ideas in \cite{PetermichlPott}. Without
loss of generality we may assume both $w$ and $w^{-1}$ are bounded so long
as these bounds do not enter into our estimates. First note that 
\begin{equation*}
\left\Vert S_{\pi }\phi \right\Vert _{L^{2}\left( w\right)
}^{2}=\sum\limits_{I\in \mathcal{D}}\left\vert \widehat{\phi }\left( I\right)
\right\vert^{2}\,\left\langle w\right\rangle _{\pi I}=\left\langle \widetilde{D}%
_{w}\phi ,\phi \right\rangle _{L^{2}}\ ,
\end{equation*}%
where $\widetilde{D}_{w}:L^{2}\rightarrow L^{2}$ is the `discrete
multiplier' map that sends $h_{I}\longrightarrow \left\langle w\right\rangle
_{\pi I}h_{I}$. We also have%
\begin{equation*}
\left\Vert \phi \right\Vert _{L^{2}\left( w\right) }^{2}=\left\langle
M_{w}\phi ,\phi \right\rangle _{L^{2}}\ ,
\end{equation*}%
where $M_{w}:L^{2}\rightarrow L^{2}$ is the operator of pointwise
multiplication by $w$. We first claim the operator inequality%
\begin{equation*}
M_{w}\lesssim \left[ w\right] _{A_{2}}\widetilde{D}_{w}\ ,
\end{equation*}%
which upon taking inverses, is equivalent to 
\begin{equation}
\widetilde{D}_{w}^{-1}\lesssim \left[ w\right] _{A_{2}}\left( M_{w}\right)
^{-1}=\left[ w\right] _{A_{2}}M_{w^{-1}}\ ,  \label{same as}
\end{equation}%
where $\widetilde{D}_{w}^{-1}:L^{2}\rightarrow L^{2}$ is the map that sends $%
h_{I}\longrightarrow \frac{1}{\left\langle w\right\rangle _{\pi I}}h_{I}$.
But, \eqref{same as} is equivalent to%
\begin{equation*}
\sum\limits_{I\in \mathcal{D}}\left\vert \widehat{\phi }\left( I\right) \right\vert^{2}\frac{1}{%
\left\langle w\right\rangle _{\pi I}}\lesssim \left[ w\right]
_{A_{2}}\left\Vert \phi \right\Vert _{L^{2}\left( w^{-1}\right) }^{2}.
\end{equation*}%
At this point we simply observe that since $\left\vert \pi I\right\vert
=2\left\vert I\right\vert $ and $w\left( I\right) \leq w\left( \pi I\right) $%
, we have 
\begin{equation*}
\sum\limits_{I\in \mathcal{D}}\left\vert \widehat{\phi }\left( I\right)\right\vert ^{2}\frac{1}{%
\left\langle w\right\rangle _{\pi I}}\leq 2\sum\limits_{I\in \mathcal{D}}%
\left\vert \widehat{\phi }\left( I\right)\right\vert ^{2}\frac{1}{\left\langle w\right\rangle _{I}}
\end{equation*}%
and then we invoke the following inequality in \cite{PetermichlPott}:%
\begin{equation*}
\sum\limits_{I\in \mathcal{D}}\left\vert \widehat{\phi }\left( I\right) \right\vert^{2}\frac{1}{%
\left\langle w\right\rangle _{I}}\lesssim \left[ w\right] _{A_{2}}\left\Vert
\phi \right\Vert _{L^{2}\left( w^{-1}\right) }^{2},
\end{equation*}%
and this proves \eqref{same as}.

Now we apply a duality argument to obtain%
\begin{equation}
\widetilde{D}_{w}\lesssim \left[ w\right] _{A_{2}}^{2}M_{w}\ ,
\label{duality argument}
\end{equation}%
which is equivalent to the desired inequality in the statement of the
theorem. To see \eqref{duality argument}, we note that%
\begin{equation*}
\widetilde{D}_{w}\leq \left[ w\right] _{A_{2}}\left( \widetilde{D}%
_{w^{-1}}\right) ^{-1}\ ,
\end{equation*}%
follows from%
\begin{eqnarray*}
\left\langle \widetilde{D}_{w}\phi ,\phi \right\rangle _{L^{2}}
&=&\sum\limits_{I\in \mathcal{D}}\left\vert \widehat{\phi }\left( I\right)
\right\vert^{2}\,\left\langle w\right\rangle _{\pi I}=\sum\limits_{I\in \mathcal{D}}\left\vert 
\widehat{\phi }\left( I\right) \right\vert^{2}\frac{\left\langle w\right\rangle _{\pi
I}\left\langle w^{-1}\right\rangle _{\pi I}}{\left\langle
w^{-1}\right\rangle _{\pi I}} \\
&\leq &\left[ w\right] _{A_{2}}\sum\limits_{I\in \mathcal{D}}\left\vert \widehat{\phi }%
\left( I\right) \right\vert^{2}\frac{1}{\left\langle w^{-1}\right\rangle _{\pi I}}=%
\left[ w\right] _{A_{2}}\left\langle \left( \widetilde{D}_{w^{-1}}\right)
^{-1}\phi ,\phi \right\rangle _{L^{2}}.
\end{eqnarray*}%
Now we continue by applying \eqref{same as} with $w$ replaced by $w^{-1}$ to
get%
\begin{equation*}
\widetilde{D}_{w}\leq \left[ w\right] _{A_{2}}\left( \widetilde{D}%
_{w^{-1}}\right) ^{-1}\leq \left[ w\right] _{A_{2}}\left[ w^{-1}\right]
_{A_{2}}\left( M_{w^{-1}}\right) ^{-1}=\left[ w\right] _{A_{2}}^{2}M_{w}\ .
\end{equation*}
\end{proof}

Again, using this Theorem \ref{ModifiedSquare} with the function $\phi=w^{-%
\frac{1}{2}}\mathbf{1}_{I}$ for $I\in\mathcal{D}$ yields

\begin{equation}  \label{ModSquareFunctionEst}
\sum_{J\subset I} \widehat{w^{-\frac{1}{2}}}(J)^2\left\langle
w\right\rangle_{\pi J}\lesssim \left[w\right]^2_{A_2}\vert I\vert\quad
\forall I\in\mathcal{D}
\end{equation}
which will be used frequently in the proof below.

\section{Proof of Theorem \protect\ref{linearA_2_pararesolution}}

Our method of attack on Theorem \ref{linearA_2_pararesolution} will use the language of paraproducts. We expand
the operator $M_{w^{\frac{1}{2}}}\mathcal{S}M_{w^{-\frac{1}{2}}}$ in term of
the canonical paraproducts with symbols $w^{\frac{1}{2}}$ and $w^{-\frac{1}{2%
}}$ to obtain, 
\begin{equation}
\left( \mathsf{P}_{\widehat{w^{\frac{1}{2}}}}^{(0,1)}+\mathsf{P}_{\widehat{%
w^{\frac{1}{2}}}}^{(1,0)}+\mathsf{P}_{\langle w^{\frac{1}{2}}\rangle
}^{(0,0)}\right) \mathcal{S}\left( \mathsf{P}_{\widehat{w^{-\frac{1}{2}}}%
}^{(0,1)}+\mathsf{P}_{\widehat{w^{-\frac{1}{2}}}}^{(1,0)}+\mathsf{P}%
_{\langle w^{-\frac{1}{2}}\rangle }^{(0,0)}\right) .  \label{Prod_Expan}
\end{equation}%
This results in nine terms and we will study each of these separately to see
that they are controlled by (no worse than) the linear characteristic of the
weight. Doing so will provide a proof of Theorem \ref{linearA_2_pararesolution}.

In what follows, any time there are `internal' zeros, the Haar shift $%
\mathcal{S}$ can be absorbed, and we are left with simply studying a
modified (shifted) paraproduct.

\subsection{Estimating Easy Terms}

There are four easy terms that arise from \eqref{Prod_Expan}, and they are
easy because the composition of the paraproducts reduce to classical
paraproduct type operators. To motivate our approach, we point out that
simple computations (that absorb the Haar shift and are given explicitly
below) give 
\begin{equation}
\mathsf{P}_{\widehat{w^{\frac{1}{2}}}}^{(1,0)}\mathcal{S}\mathsf{P}_{%
\widehat{w^{-\frac{1}{2}}}}^{(0,1)}\approx \mathsf{P}_{\widehat{w^{\frac{1}{2%
}}}\circ \widehat{w^{-\frac{1}{2}}}}^{(1,1)};  \label{Easy1}
\end{equation}

\begin{equation}  \label{Easy2}
\mathsf{P}_{\widehat{w^{\frac{1}{2}}}}^{(1,0)} \mathcal{S} \mathsf{P}%
_{\langle w^{-\frac{1}{2}}\rangle}^{(0,0)}\approx \mathsf{P}_{\widehat{w^{%
\frac{1}{2}}}\circ \langle w^{-\frac{1}{2}}\rangle}^{(1,0)};
\end{equation}

\begin{equation}  \label{Easy3}
\mathsf{P}_{\langle w^{\frac{1}{2}}\rangle}^{(0,0)}\mathcal{S}\mathsf{P}_{%
\widehat{w^{-\frac{1}{2}}}}^{(0,1)}\approx \mathsf{P}_{\langle w^{\frac{1}{2}%
}\rangle\circ \widehat{w^{-\frac{1}{2}}}}^{(0,1)};
\end{equation}

\begin{equation}  \label{Easy4}
\mathsf{P}_{\langle w^{\frac{1}{2}}\rangle}^{(0,0)}\mathcal{S}\mathsf{P}%
_{\langle w^{-\frac{1}{2}}\rangle}^{(0,0)}\approx \mathsf{P}_{\langle w^{%
\frac{1}{2}}\rangle\circ \langle w^{-\frac{1}{2}}\rangle}^{(0,0)}.
\end{equation}

Here $a\circ b$ is the Schur product of the sequences $a$ and $b$, and we
are letting $\approx $ mean that there has been a shift in one of the basis
elements appearing in the definition of the Haar basis (this presents no
problem in the analysis of these operators since it simply results in a
change in the absolute constants appearing; we will make this rigorous
shortly). Notice that \eqref{Easy2}, \eqref{Easy3} and \eqref{Easy4} are of
the type considered in Lemma \ref{ClassicalCharacterization} and so have a
complete characterization, while \eqref{Easy1} can be handled by Lemma \ref%
{ClassicalCharacterization2}. In each of these characterizations we are of
course left with showing that the norm on the symbol for these classical
operators can be controlled by the linear power of the characteristic of the
weight. We now make these ideas precise.

\subsubsection{Estimating Term \eqref{Easy1}}

Note that a simple computation shows that 
\begin{equation}
\mathsf{P}_{\widehat{w^{\frac{1}{2}}}}^{(1,0)}\mathcal{S}\mathsf{P}_{%
\widehat{w^{-\frac{1}{2}}}}^{(0,1)}=\sum_{I\in \mathcal{D}}\widehat{w^{\frac{%
1}{2}}}(I_{-})\widehat{w^{-\frac{1}{2}}}(I)h_{I_{-}}^{1}\otimes h_{I}^{1}. 
\notag
\end{equation}%
By the Carleson Embedding Theorem we see that 
\begin{eqnarray*}
\left\vert \left\langle \mathsf{P}_{\widehat{w^{\frac{1}{2}}}}^{(1,0)}%
\mathcal{S}\mathsf{P}_{\widehat{w^{-\frac{1}{2}}}}^{(0,1)}\phi ,\psi
\right\rangle _{L^{2}}\right\vert &=&\left\vert \sum_{I\in \mathcal{D}}%
\widehat{w^{\frac{1}{2}}}(I_{-})\widehat{w^{-\frac{1}{2}}}(I)\left\langle
\phi ,h_{I}^{1}\right\rangle _{L^{2}}\left\langle \psi
,h_{I_{-}}^{1}\right\rangle _{L^{2}}\right\vert \\
&\leq &\sum_{I\in \mathcal{D}}\left\vert \widehat{w^{\frac{1}{2}}}(I_{-})%
\widehat{w^{-\frac{1}{2}}}(I)\left\langle \phi ,h_{I}^{1}\right\rangle
_{L^{2}}\left\langle \psi ,h_{I_{-}}^{1}\right\rangle _{L^{2}}\right\vert \\
&\leq &\left( \sum_{I\in \mathcal{D}}\left\vert \widehat{w^{\frac{1}{2}}}%
(I_{-})\widehat{w^{-\frac{1}{2}}}(I) \left\langle \phi\right\rangle_I^2
\right\vert \right) ^{\frac{1}{2}}\left( \sum_{I\in \mathcal{D}}\left\vert 
\widehat{w^{\frac{1}{2}}}(I_{-})\widehat{w^{-\frac{1}{2}}}(I) \left\langle
\psi\right\rangle_{I_-}^2\right\vert \right) ^{\frac{1}{2}} \\
&\lesssim &\left\Vert \phi \right\Vert _{L^{2}}\left\Vert \psi \right\Vert
_{L^{2}}\sup_{J\in \mathcal{D}}\frac{1}{\left\vert J\right\vert }%
\sum_{I\subset J}\left\vert \widehat{w^{\frac{1}{2}}}(I_{-})\widehat{w^{-%
\frac{1}{2}}}(I)\right\vert .
\end{eqnarray*}%
However, by Cauchy-Schwarz we then have 
\begin{eqnarray*}
\frac{1}{\left\vert J\right\vert }\sum_{I\subset J}\left\vert \widehat{w^{%
\frac{1}{2}}}(I_{-})\widehat{w^{-\frac{1}{2}}}(I)\right\vert &\leq &\frac{1}{%
\left\vert J\right\vert }\left\Vert w^{\frac{1}{2}}\mathbf{1}_{J}\right\Vert
_{L^{2}}\left\Vert w^{-\frac{1}{2}}\mathbf{1}_{J}\right\Vert _{L^{2}} \\
&=&\left( \left\langle w\right\rangle _{J}\left\langle w^{-1}\right\rangle
_{J}\right) ^{\frac{1}{2}}\leq \left[ w\right] _{A_{2}}^{\frac{1}{2}},
\end{eqnarray*}%
which gives 
\begin{equation*}
\left\Vert \mathsf{P}_{\widehat{w^{\frac{1}{2}}}}^{(1,0)}\mathcal{S}\mathsf{P%
}_{\widehat{w^{-\frac{1}{2}}}}^{(0,1)}\right\Vert _{L^{2}\rightarrow
L^{2}}\lesssim \left[ w\right] _{A_{2}}^{\frac{1}{2}}\leq \left[ w\right]
_{A_{2}},
\end{equation*}%
yielding the desired linear bound for term \eqref{Easy1}.

\subsubsection{Estimating Terms \eqref{Easy2} and \eqref{Easy3}}

These terms have are symmetric and so we focus only on the first term. In
this case term \eqref{Easy2} reduces to a classical paraproduct operator
with a symbol given by the product of two sequences as indicated by the
notation $\approx $. Indeed, we have the identity, 
\begin{equation}
\mathsf{P}_{\widehat{w^{\frac{1}{2}}}}^{(1,0)}\mathcal{S}\mathsf{P}_{\langle
w^{-\frac{1}{2}}\rangle }^{(0,0)}=\sum_{I\in \mathcal{D}}\widehat{w^{\frac{1%
}{2}}}(I_{-})\langle w^{-\frac{1}{2}}\rangle _{I}h_{I_{-}}^{1}\otimes h_{I}.
\notag
\end{equation}%
Now by the Carleson Embedding Theorem applied to $\psi $ we have that%
\begin{eqnarray*}
\left\vert \left\langle \mathsf{P}_{\widehat{w^{\frac{1}{2}}}}^{(1,0)}%
\mathcal{S}\mathsf{P}_{\langle w^{-\frac{1}{2}}\rangle }^{(0,0)}\phi ,\psi
\right\rangle _{L^{2}}\right\vert &=&\left\vert \sum_{I\in \mathcal{D}}%
\widehat{w^{\frac{1}{2}}}(I_{-})\langle w^{-\frac{1}{2}}\rangle
_{I}\left\langle \phi ,h_{I}\right\rangle _{L^{2}}\left\langle \psi
,h_{I_{-}}^{1}\right\rangle _{L^{2}}\right\vert \\
&\leq &\left\Vert \phi \right\Vert _{L^{2}}\left( \sum_{I\in \mathcal{D}%
}\left\vert \widehat{w^{\frac{1}{2}}}(I_{-})\langle w^{-\frac{1}{2}}\rangle
_{I}\right\vert ^{2} \left\langle \psi\right\rangle_{I_-}^2\right) ^{\frac{1%
}{2}} \\
&\lesssim &\left\Vert \phi \right\Vert _{L^{2}}\left\Vert \psi \right\Vert
_{L^{2}}\sup_{J\in \mathcal{D}}\frac{1}{\left\vert J\right\vert }%
\sum_{I\subset J}\left\vert \widehat{w^{\frac{1}{2}}}(I_{-})\langle w^{-%
\frac{1}{2}}\rangle _{I}\right\vert ^{2}.
\end{eqnarray*}%
But 
\begin{equation*}
\sup_{J\in \mathcal{D}}\frac{1}{|J|}\sum_{I\subset J}\left\vert \widehat{w^{%
\frac{1}{2}}}(I_{-})\langle w^{-\frac{1}{2}}\rangle _{I}\right\vert ^{2}\leq
\sup_{J\in \mathcal{D}}\frac{1}{|J|}\sum_{I\subset J}\widehat{w^{\frac{1}{2}}%
}(I_{-})^{2}\langle w^{-1}\rangle _{I}\lesssim \left[ w\right] _{A_{2}}^{2},
\end{equation*}%
where we have used the linear bound for the square function in %
\eqref{ModSquareFunctionEst} in the last inequality. Thus, 
\begin{equation*}
\left\Vert \mathsf{P}_{\widehat{w^{\frac{1}{2}}}}^{(1,0)}\mathcal{S}\mathsf{P%
}_{\langle w^{-\frac{1}{2}}\rangle }^{(0,0)}\right\Vert _{L^{2}\rightarrow
L^{2}}\lesssim \left[ w\right] _{A_{2}}
\end{equation*}%
which gives the desired linear bound in terms of the $A_{2}$ characteristic.

\subsubsection{ Estimating Term \eqref{Easy4}}

Term \eqref{Easy4} reduces to a standard Haar multiplier. Indeed, we have
the identity 
\begin{equation}
\mathsf{P}_{\langle w^{\frac{1}{2}}\rangle }^{(0,0)}\mathcal{S}\mathsf{P}%
_{\langle w^{-\frac{1}{2}}\rangle }^{(0,0)}=\sum_{I\in \mathcal{D}}\langle
w^{\frac{1}{2}}\rangle _{I}\langle w^{-\frac{1}{2}}\rangle
_{I_{-}}h_{I_{-}}\otimes h_{I},  \notag
\end{equation}%
where $\left( h_{J}\otimes h_{K}\right) f\equiv \left\langle
f,h_{K}\right\rangle h_{J}$, and then 
\begin{equation*}
\left\Vert \mathsf{P}_{\langle w^{\frac{1}{2}}\rangle }^{(0,0)}\mathcal{S}%
\mathsf{P}_{\langle w^{-\frac{1}{2}}\rangle }^{(0,0)}\right\Vert
_{L^{2}\rightarrow L^{2}}=\sup_{I\in \mathcal{D}}\langle w^{\frac{1}{2}%
}\rangle _{I}\langle w^{-\frac{1}{2}}\rangle _{I_{-}}\lesssim \sup_{I\in 
\mathcal{D}}\left\langle w\right\rangle _{I}^{\frac{1}{2}}\left\langle
w^{-1}\right\rangle _{I}^{\frac{1}{2}}=\left[ w\right] _{A_{2}}^{\frac{1}{2}%
}\leq \left[ w\right] _{A_{2}}.
\end{equation*}%
This gives the desired linear estimate for \eqref{Easy4}.

\subsection{Estimating Hard Terms}

There are five remaining terms to be controlled. These include the four
difficult terms,

\begin{eqnarray}
\mathsf{P}_{\widehat{w^{\frac{1}{2}}}}^{(0,1)}\mathcal{S}\mathsf{P}_{%
\widehat{w^{-\frac{1}{2}}}}^{(0,1)} & \approx & \mathsf{P}_{\widehat{w^{%
\frac{1}{2}}}}^{(0,1)} \mathsf{P}_{\widehat{w^{-\frac{1}{2}}}}^{(0,1)};
\label{Difficult1} \\
\mathsf{P}_{\widehat{w^{\frac{1}{2}}}}^{(0,1)}\mathcal{S}\mathsf{P}_{\langle
w^{-\frac{1}{2}}\rangle}^{(0,0)} & \approx & \mathsf{P}_{\widehat{w^{\frac{1%
}{2}}}}^{(0,1)} \mathsf{P}_{\langle w^{-\frac{1}{2}}\rangle}^{(0,0)};
\label{Difficult2} \\
\mathsf{P}_{\widehat{w^{\frac{1}{2}}}}^{(1,0)}\mathcal{S}\mathsf{P}_{%
\widehat{w^{-\frac{1}{2}}}}^{(1,0)} & \approx & \mathsf{P}_{\widehat{w^{%
\frac{1}{2}}}}^{(1,0)} \mathsf{P}_{\widehat{w^{-\frac{1}{2}}}}^{(1,0)};
\label{Difficult3} \\
\mathsf{P}_{\langle w^{\frac{1}{2}}\rangle}^{(0,0)}\mathcal{S}\mathsf{P}_{%
\widehat{w^{-\frac{1}{2}}}}^{(1,0)} & \approx & \mathsf{P}_{\langle w^{\frac{%
1}{2}}\rangle}^{(0,0)} \mathsf{P}_{\widehat{w^{-\frac{1}{2}}}}^{(1,0)}.
\label{Difficult4}
\end{eqnarray}

To estimate terms \eqref{Difficult1} and \eqref{Difficult3} we will rely on
disbalanced Haar functions adapted to the weight $w$ and $w^{-1}$. For these
terms we will also proceed by computing the norm of the operators in
question by using duality. Key to this will be the application of the
Carleson Embedding Theorem. The proof of the necessary estimates for these
terms are carried out in subsection \ref{subsec:Difficult13}. Terms %
\eqref{Difficult2} and \eqref{Difficult4} will be handled via a similar
method; their analysis is handled in subsection \ref{subsec:Difficult24}.

The remaining very difficult term is the one for which the Haar shift can
not be absorbed into one of the paraproducts. Namely, we need to control the
following expression, 
\begin{equation}
\mathsf{P}_{\widehat{w^{\frac{1}{2}}}}^{(0,1)}\mathcal{S}\mathsf{P}_{%
\widehat{w^{-\frac{1}{2}}}}^{(1,0)}.  \label{VeryDifficult}
\end{equation}%
In this situation, we will explicitly compute $\left\langle \mathcal{S}%
h_{J}^{1},h_{L}^{1}\right\rangle _{L^{2}}$ and see that the shift presents
no problem in how the operator behaves, nor in estimating its norm in terms
of the $A_{2}$ characteristic $\left[ w\right] _{A_{2}}$. This is carried
out in subsection \ref{subsec:VeryDifficult}.

\subsubsection{Estimating Terms \eqref{Difficult1} and \eqref{Difficult3}}

\label{subsec:Difficult13}

Once again these two terms are symmetric and so it suffices to prove the
desired estimate only for \eqref{Difficult1}. We need to show that 
\begin{equation}
\left\Vert \mathsf{P}_{\widehat{w^{\frac{1}{2}}}}^{(0,1)}\mathcal{S}\mathsf{P%
}_{\widehat{w^{-\frac{1}{2}}}}^{(0,1)}\right\Vert _{L^{2}}\lesssim \left[ w%
\right] _{A_{2}}.
\end{equation}%
Proceeding by duality, we fix $\phi ,\psi \in L^{2}$ and consider 
\begin{eqnarray*}
\left\langle \mathsf{P}_{\widehat{w^{\frac{1}{2}}}}^{(0,1)}\mathcal{S}%
\mathsf{P}_{\widehat{w^{-\frac{1}{2}}}}^{(0,1)}\phi ,\psi \right\rangle
_{L^{2}} &=&\left\langle \mathcal{S}\mathsf{P}_{\widehat{w^{-\frac{1}{2}}}%
}^{(0,1)}\phi ,\mathsf{P}_{\widehat{w^{\frac{1}{2}}}}^{(1,0)}\psi
\right\rangle _{L^{2}} \\
&=&\sum_{J,K\in \mathcal{D}}\widehat{w^{-\frac{1}{2}}}(J)\left\langle \phi
\right\rangle _{J}\widehat{w^{\frac{1}{2}}}(K)\widehat{\psi }(K)\left\langle
h_{J_{-}},h_{K}^{1}\right\rangle _{L^{2}} \\
&=&\sum_{J\in \mathcal{D}}\widehat{w^{-\frac{1}{2}}}(J)\left\langle \phi
\right\rangle _{J}\left( \sum_{K\in \mathcal{D}}\widehat{w^{\frac{1}{2}}}(K)%
\widehat{\psi }(K)\left\langle h_{J_{-}},h_{K}^{1}\right\rangle
_{L^{2}}\right) \\
&=&\sum_{J\in \mathcal{D}}\widehat{w^{-\frac{1}{2}}}(J)\left\langle \phi
\right\rangle _{J}\left( \frac{1}{\left\vert J_{-}\right\vert ^{\frac{1}{2}}}%
\sum_{K\subsetneq J_{-}}\widehat{w^{\frac{1}{2}}}(K)\widehat{\psi }(K)\delta
\left( K,J_{-}\right) \right) \\
&=&\sum_{J\in \mathcal{D}}\widehat{w^{-\frac{1}{2}}}(J)\left\langle \phi
\right\rangle _{J}\left( \widehat{\psi w^{\frac{1}{2}}}(J_{-})-\widehat{w^{%
\frac{1}{2}}}(J_{-})\left\langle \psi \right\rangle _{J_{-}}-\widehat{\psi }%
(J_{-})\left\langle w^{\frac{1}{2}}\right\rangle _{J_{-}}\right) \\
&\equiv &T_{1}+T_{2}+T_{3}.
\end{eqnarray*}%
The fifth equality follows by an application of the product formula for Haar
coefficients. We need to show that each of these terms has the desired
estimate 
\begin{equation*}
\left\vert T_{j}\right\vert \lesssim \left[ w\right] _{A_{2}}\left\Vert \phi
\right\Vert _{L^{2}}\left\Vert \psi \right\Vert _{L^{2}}
\end{equation*}%
since this will imply that 
\begin{equation*}
\left\Vert \mathsf{P}_{\widehat{w^{\frac{1}{2}}}}^{(0,1)}\mathcal{S}\mathsf{P%
}_{\widehat{w^{-\frac{1}{2}}}}^{(0,1)}\right\Vert _{L^{2}\rightarrow
L^{2}}=\sup_{\phi ,\psi \in L^{2}}\left\vert \left\langle \mathsf{P}_{%
\widehat{w^{\frac{1}{2}}}}^{(0,1)}\mathcal{S}\mathsf{P}_{\widehat{w^{-\frac{1%
}{2}}}}^{(0,1)}\phi ,\psi \right\rangle _{L^{2}}\right\vert \lesssim \left[ w%
\right] _{A_{2}}
\end{equation*}%
which implies the desired estimate on the norm of the operator.

Consider term $T_{3}$. This term can be controlled by 
\begin{eqnarray*}
\left\vert T_{3}\right\vert &\leq &\left\vert \sum_{J\in \mathcal{D}}%
\widehat{w^{-\frac{1}{2}}}(J)\left\langle \phi \right\rangle _{J}\widehat{%
\psi }(J_{-})\left\langle w^{\frac{1}{2}}\right\rangle _{J_{-}}\right\vert \\
&\leq &\left\Vert \psi \right\Vert _{L^{2}}\left( \sum_{J\in \mathcal{D}}%
\widehat{w^{-\frac{1}{2}}}(J)^{2}\left\langle w^{\frac{1}{2}}\right\rangle
_{J_{-}}^{2}\left\langle \phi \right\rangle _{J_{-}}^{2}\right) ^{\frac{1}{2}%
} \\
&\lesssim &\left[ w\right] _{A_{2}}\left\Vert \phi \right\Vert
_{L^{2}}\left\Vert \psi \right\Vert _{L^{2}},
\end{eqnarray*}%
with the last estimate following by an application of the Carleson Embedding
Theorem using \eqref{SquareFunctionEst2}. For the term $T_{2}$, we have 
\begin{eqnarray*}
\left\vert T_{2}\right\vert &\leq &\left\vert \sum_{J\in \mathcal{D}}%
\widehat{w^{-\frac{1}{2}}}(J)\widehat{w^{\frac{1}{2}}}(J_{-})\left\langle
\psi \right\rangle _{J}\left\langle \phi \right\rangle _{J_{-}}\right\vert \\
&\leq &\left( \sum_{J\in \mathcal{D}}\left\vert \widehat{w^{-\frac{1}{2}}}(J)%
\widehat{w^{\frac{1}{2}}}(J_{-})\right\vert \left\langle \psi \right\rangle
_{J}^{2}\right) ^{\frac{1}{2}}\left( \sum_{J\in \mathcal{D}}\left\vert 
\widehat{w^{-\frac{1}{2}}}(J)\widehat{w^{\frac{1}{2}}}(J_{-})\right\vert
\left\langle \phi \right\rangle _{J_{-}}^{2}\right) ^{\frac{1}{2}} \\
&\lesssim &\left[ w\right] _{A_{2}}^{\frac{1}{2}}\left\Vert \phi \right\Vert
_{L^{2}}\left\Vert \psi \right\Vert _{L^{2}}.
\end{eqnarray*}%
Again we have used here two applications of the Carleson Embedding Theorem,
one for $\phi $ and one for $\psi $, since we have that 
\begin{eqnarray*}
\sum_{K\subset J}\left\vert \widehat{w^{-\frac{1}{2}}}(K)\widehat{w^{\frac{1%
}{2}}}(K_{-})\right\vert &\leq &\left( \sum_{K\subset J}\widehat{w^{-\frac{1%
}{2}}}(K)^{2}\right) ^{\frac{1}{2}}\left( \sum_{K\subset J}\widehat{w^{\frac{%
1}{2}}}(K_{-})^{2}\right) ^{\frac{1}{2}} \\
&\leq &\sqrt{w^{-1}(J)w(J)}=\sqrt{\frac{w^{-1}(J)w(J)}{\left\vert
J\right\vert ^{2}}}\left\vert J\right\vert \\
&\leq &\left[ w\right] _{A_{2}}^{\frac{1}{2}}\left\vert J\right\vert .
\end{eqnarray*}

Finally, we prove the estimate for term $T_{1}$, which requires the use of
disbalanced Haar functions. We expand the terms in $T_{1}$ using Haar
functions with respect to two different disbalanced bases. In particular,
using \eqref{disbalanced_defs2} we have 
\begin{eqnarray*}
\left\langle \psi w^{\frac{1}{2}},h_{J_{-}}\right\rangle _{L^{2}}
&=&C_{J_{-}}(w)\left\langle \psi w^{\frac{1}{2}},h_{J_{-}}^{w}\right\rangle
_{L^{2}}+D_{J_{-}}(w)\left\langle \psi w^{\frac{1}{2}}\right\rangle _{J_{-}}
\\
&=&C_{J_{-}}(w)\left\langle \psi w^{-\frac{1}{2}},h_{J_{-}}^{w}\right\rangle
_{L^{2}(w)}+D_{J_{-}}(w)\left\langle w\right\rangle _{J_{-}}\mathbb{E}%
_{J_{-}}^{w}\left( \psi w^{-\frac{1}{2}}\right)
\end{eqnarray*}%
and 
\begin{eqnarray*}
\left\langle w^{-\frac{1}{2}},h_{J}\right\rangle _{L^{2}}
&=&C_{J}(w^{-1})\left\langle w^{-\frac{1}{2}},h_{J}^{w^{-1}}\right\rangle
_{L^{2}}+D_{J}(w^{-1})\left\langle w^{-\frac{1}{2}}\right\rangle _{J} \\
&=&C_{J}(w^{-1})\left\langle w^{\frac{1}{2}},h_{J}^{w^{-1}}\right\rangle
_{L^{2}(w^{-1})}+D_{J}(w^{-1})\left\langle w^{-\frac{1}{2}}\right\rangle
_{J}.
\end{eqnarray*}%
Then we can write the term $T_{1}$ as a sum of four terms, namely 
\begin{equation*}
T_{1}=S_{1}+S_{2}+S_{3}+S_{4}
\end{equation*}%
with 
\begin{eqnarray*}
S_{1} &=&\sum_{K\in \mathcal{D}}\left\langle \phi \right\rangle
_{K}C_{K_{-}}(w)C_{K}(w^{-1})\left\langle w^{\frac{1}{2}},h_{K}^{w^{-1}}%
\right\rangle _{L^{2}(w^{-1})}\left\langle \psi w^{-\frac{1}{2}%
},h_{K_{-}}^{w}\right\rangle _{L^{2}(w)} \\
S_{2} &=&\sum_{K\in \mathcal{D}}\left\langle \phi \right\rangle
_{K}C_{K_{-}}(w)D_{K}(w^{-1})\left\langle \psi w^{-\frac{1}{2}%
},h_{K_{-}}^{w}\right\rangle _{L^{2}(w)}\left\langle w^{-\frac{1}{2}%
}\right\rangle _{K} \\
S_{3} &=&\sum_{K\in \mathcal{D}}\left\langle \phi \right\rangle
_{K}D_{K_{-}}(w)C_{K}(w^{-1})\left\langle w\right\rangle _{K_{-}}\mathbb{E}%
_{K_{-}}^{w}\left( \psi w^{-\frac{1}{2}}\right) \left\langle w^{\frac{1}{2}%
},h_{K}^{w^{-1}}\right\rangle _{L^{2}(w^{-1})} \\
S_{4} &=&\sum_{K\in \mathcal{D}}\left\langle \phi \right\rangle
_{K}D_{K_{-}}(w)D_{K}(w^{-1})\left\langle w^{-\frac{1}{2}}\right\rangle
_{K}\left\langle w\right\rangle _{K_{-}}\mathbb{E}_{K_{-}}^{w}\left( \psi
w^{-\frac{1}{2}}\right) .
\end{eqnarray*}%
We need to show that each term can be estimated by a constant times $\left[ w%
\right] _{A_{2}}\left\Vert \phi \right\Vert _{2}\left\Vert \psi \right\Vert
_{2}$, which would then imply 
\begin{equation*}
\left\vert T_{1}\right\vert \lesssim \left[ w\right] _{A_{2}}\left\Vert \phi
\right\Vert _{2}\left\Vert \psi \right\Vert _{2}
\end{equation*}%
as required. We now proceed to prove the necessary estimates.

Consider the term $S_{1}$. We have 
\begin{eqnarray*}
\left\vert S_{1}\right\vert &\leq &\left\vert \sum_{K\in \mathcal{D}%
}\left\langle \phi \right\rangle _{K}C_{K_{-}}(w)C_{K}(w^{-1})\left\langle
w^{\frac{1}{2}},h_{K}^{w^{-1}}\right\rangle _{L^{2}(w^{-1})}\left\langle
\psi w^{-\frac{1}{2}},h_{K_{-}}^{w}\right\rangle _{L^{2}(w)}\right\vert \\
&\leq &\left\Vert \psi w^{-\frac{1}{2}}\right\Vert _{L^{2}(w)}\left(
\sum_{K\in \mathcal{D}}\left\langle \phi \right\rangle
_{K}^{2}C_{K_{-}}(w)^{2}C_{K}(w^{-1})^{2}\left\langle w^{-\frac{1}{2}%
},h_{K}^{w}\right\rangle _{L^{2}(w)}^{2}\right) ^{\frac{1}{2}} \\
&=&\left\Vert \psi \right\Vert _{L^{2}}\left( \sum_{K\in \mathcal{D}%
}\left\langle \phi \right\rangle _{K}^{2}\frac{\left\langle w\right\rangle
_{K_{-+}}\left\langle w\right\rangle _{K_{--}}}{\left\langle w\right\rangle
_{K_{-}}}\frac{\left\langle w^{-1}\right\rangle _{K_{+}}\left\langle
w^{-1}\right\rangle _{K_{-}}}{\left\langle w^{-1}\right\rangle _{K}}%
\left\langle w^{\frac{1}{2}},h_{K}^{w^{-1}}\right\rangle
_{L^{2}(w^{-1})}^{2}\right) ^{\frac{1}{2}} \\
&\lesssim &\left[ w\right] _{A_{2}}^{\frac{1}{2}}\left\Vert \psi \right\Vert
_{L^{2}}\left( \sum_{K\in \mathcal{D}}\left\langle \phi \right\rangle
_{K}^{2}\left\langle w^{\frac{1}{2}},h_{K}^{w^{-1}}\right\rangle
_{L^{2}(w^{-1})}^{2}\right) ^{\frac{1}{2}} \\
&\lesssim &\left[ w\right] _{A_{2}}^{\frac{1}{2}}\left\Vert \phi \right\Vert
_{L^{2}}\left\Vert \psi \right\Vert _{L^{2}}.
\end{eqnarray*}%
Here the last inequality follows by the Carleson Embedding Theorem since 
\begin{equation*}
\sum_{K\subset J}\left\langle w^{\frac{1}{2}},h_{K}^{w^{-1}}\right\rangle
_{L^{2}(w^{-1})}^{2}\leq \left\vert J\right\vert \quad \forall J\in \mathcal{%
D}.
\end{equation*}%
Turning to term $S_{2}$ we have%
\begin{eqnarray*}
\left\vert S_{2}\right\vert &=&\left\vert \sum_{K\in \mathcal{D}%
}\left\langle \phi \right\rangle _{K}C_{K_{-}}(w)D_{K}(w^{-1})\left\langle
\psi w^{-\frac{1}{2}},h_{J_{-}}^{w}\right\rangle _{L^{2}(w)}\left\langle w^{-%
\frac{1}{2}}\right\rangle _{K}\right\vert \\
&\leq &\left( \sum_{K\in \mathcal{D}}\left\langle \psi w^{-\frac{1}{2}%
},h_{K_{-}}^{w}\right\rangle _{L^{2}(w)}^{2}\frac{\left\langle
w\right\rangle _{K_{-+}}\left\langle w\right\rangle _{K_{--}}}{\left\langle
w\right\rangle _{K_{-}}}\left\langle w^{-\frac{1}{2}}\right\rangle
_{K}^{2}\right) ^{\frac{1}{2}}\left( \sum_{K\in \mathcal{D}}\frac{\widehat{%
w^{-1}}(K)^{2}}{\left\langle w^{-1}\right\rangle _{K}^{2}}\left\langle \phi
\right\rangle _{K}^{2}\right) ^{\frac{1}{2}} \\
&\lesssim &\left[ w\right] _{A_{2}}^{\frac{1}{2}}\left\Vert \psi w^{-\frac{1%
}{2}}\right\Vert _{L^{2}(w)}\left( \sum_{K\in \mathcal{D}}\frac{\widehat{%
w^{-1}}(K)^{2}}{\left\langle w^{-1}\right\rangle _{K}^{2}}\left\langle \phi
\right\rangle _{K}^{2}\right) ^{\frac{1}{2}} \\
&\leq &\left[ w\right] _{A_{2}}^{\frac{1}{2}}\left\Vert \psi \right\Vert
_{L^{2}}\left( \sum_{K\in \mathcal{D}}\frac{\widehat{w^{-1}}(K)^{2}}{%
\left\langle w^{-1}\right\rangle _{K}^{2}}\left\langle \phi \right\rangle
_{K}^{2}\right) ^{\frac{1}{2}} \\
&\leq &\left[ w\right] _{A_{2}}\left\Vert \phi \right\Vert
_{L^{2}}\left\Vert \psi \right\Vert _{L^{2}}
\end{eqnarray*}%
with the last estimate following from the Carleson Embedding Theorem, once
we prove the following estimate: 
\begin{equation}
\sum_{K\subset J}\frac{\widehat{w^{-1}}(K)^{2}}{\left\langle
w^{-1}\right\rangle _{K}^{2}}\lesssim \left[ w\right] _{A_{2}}\left\vert
J\right\vert \quad \forall J\in \mathcal{D}.  \label{InterpEst}
\end{equation}%
To prove \eqref{InterpEst} recall the following two estimates in \cite%
{Petermichl}*{Lemma 5.2 and Lemma 5.3}, translated to the notation of this
paper: 
\begin{eqnarray*}
\sum_{K\subset J}\frac{\widehat{w^{-1}}(K)^{2}}{\left\langle
w^{-1}\right\rangle _{K}} &\lesssim &\left[ w\right] _{A_{2}}w^{-1}(J)\quad
\forall J\in \mathcal{D} \\
\sum_{K\subset J}\frac{\widehat{w^{-1}}(K)^{2}}{\left\langle
w^{-1}\right\rangle _{K}^{3}} &\lesssim &w(J)\quad \forall J\in \mathcal{D}.
\end{eqnarray*}%
These two estimates coupled with a simple application of Cauchy-Schwarz then
proves \eqref{InterpEst}. Indeed, 
\begin{eqnarray*}
\sum_{K\subset J}\frac{\widehat{w^{-1}}(K)^{2}}{\left\langle
w^{-1}\right\rangle _{K}^{2}} &=&\sum_{K\subset J}\frac{\widehat{w^{-1}}%
(K)^{2}}{\left\langle w^{-1}\right\rangle _{K}}1\frac{1}{\left\langle
w^{-1}\right\rangle _{K}} \\
&\leq &\sqrt{\sum_{K\subset J}\frac{\widehat{w^{-1}}(K)^{2}}{\left\langle
w^{-1}\right\rangle _{K}}1^{2}}\ \sqrt{\sum_{K\subset J}\frac{\widehat{w^{-1}%
}(K)^{2}}{\left\langle w^{-1}\right\rangle _{K}}\frac{1}{\left\langle
w^{-1}\right\rangle _{K}^{2}}} \\
&\lesssim &\sqrt{\left[ w\right] _{A_{2}}w^{-1}(J)}\ \sqrt{w(J)} \\
&=&\sqrt{\left[ w\right] _{A_{2}}}\sqrt{w(J)w^{-1}(J)}\leq \left[ w\right]
_{A_{2}}\left\vert J\right\vert .
\end{eqnarray*}

For the term $S_{3}$ we have 
\begin{eqnarray*}
\left\vert S_{3}\right\vert &=&\left\vert \sum_{K\in \mathcal{D}%
}\left\langle \phi \right\rangle _{K}D_{K_{-}}(w)C_{K}(w^{-1})\left\langle
w\right\rangle _{K_{-}}\mathbb{E}_{K_{-}}^{w}\left( \psi w^{-\frac{1}{2}%
}\right) \left\langle w^{\frac{1}{2}},h_{K}^{w^{-1}}\right\rangle
_{L^{2}(w^{-1})}\right\vert \\
&=&\left\vert \sum_{K\in \mathcal{D}}\left\langle \phi \right\rangle _{K}%
\frac{\widehat{w}(K_{-})}{\left\langle w\right\rangle _{K_{-}}}\left\langle
w\right\rangle _{K_{-}}\mathbb{E}_{K_{-}}^{w}\left( \psi w^{-\frac{1}{2}%
}\right) \sqrt{\frac{\left\langle w^{-1}\right\rangle _{K_{+}}\left\langle
w^{-1}\right\rangle _{K_{-}}}{\left\langle w^{-1}\right\rangle _{K}}}%
\left\langle w^{\frac{1}{2}},h_{K}^{w^{-1}}\right\rangle
_{L^{2}(w^{-1})}\right\vert \\
&=&\left\vert \sum_{K\in \mathcal{D}}\left\langle \phi \right\rangle _{K}%
\widehat{w}(K_{-})\mathbb{E}_{K_{-}}^{w}\left( \psi w^{-\frac{1}{2}}\right) 
\sqrt{\frac{\left\langle w^{-1}\right\rangle _{K_{+}}\left\langle
w^{-1}\right\rangle _{K_{-}}}{\left\langle w^{-1}\right\rangle _{K}}}%
\left\langle w^{\frac{1}{2}},h_{K}^{w^{-1}}\right\rangle
_{L^{2}(w^{-1})}\right\vert \\
&\leq &\left( \sum_{K\in \mathcal{D}}\left\langle \phi \right\rangle
_{K}^{2}\left\langle w^{\frac{1}{2}},h_{K}^{w^{-1}}\right\rangle
_{L^{2}(w^{-1})}^{2}\right) ^{\frac{1}{2}}\left( \sum_{K\in \mathcal{D}}%
\frac{\left\langle w^{-1}\right\rangle _{K_{+}}\left\langle
w^{-1}\right\rangle _{K_{-}}}{\left\langle w^{-1}\right\rangle _{K}}\widehat{%
w}(K_{-})^{2}\mathbb{E}_{K}^{w}\left( \psi w^{-\frac{1}{2}}\right)
^{2}\right) ^{\frac{1}{2}} \\
&\lesssim &\left\Vert \phi \right\Vert _{L^{2}}\left( \sum_{K\in \mathcal{D}}%
\frac{\left\langle w^{-1}\right\rangle _{K_{+}}\left\langle
w^{-1}\right\rangle _{K_{-}}}{\left\langle w^{-1}\right\rangle _{K}}\widehat{%
w}(K_{-})^{2}\mathbb{E}_{K_{-}}^{w}\left( \psi w^{-\frac{1}{2}}\right)
^{2}\right) ^{\frac{1}{2}} \\
&\lesssim &\left[ w\right] _{A_{2}}\left\Vert \psi \right\Vert
_{L^{2}}\left\Vert \phi \right\Vert _{L^{2}}.
\end{eqnarray*}%
For the last inequality we used the Carleson Embedding Theorem twice, first
applied to $\phi $ as above, and then to $\psi $ upon noting that 
\begin{equation*}
\sum_{K\subset J}\frac{\left\langle w^{-1}\right\rangle _{K_{+}}\left\langle
w^{-1}\right\rangle _{K_{-}}}{\left\langle w^{-1}\right\rangle _{K}}\widehat{%
w}(K_{-})^{2}\lesssim \sum_{K\subset J}\left\langle w^{-1}\right\rangle _{K}%
\widehat{w}(K_{-})^{2}\lesssim \left[ w\right] _{A_{2}}^{2}w(J)\quad \forall
J\in \mathcal{D}
\end{equation*}%
via the linear bound for the modified square function %
\eqref{ModSquareFunctionEst}.

Finally consider the term $S_4$. We have 
\begin{eqnarray*}
\left\vert S_4 \right\vert & = & \left\vert \sum_{K\in\mathcal{D}}
\left\langle \phi\right\rangle_K D_{K_-}(w) D_K(w^{-1}) \left\langle w^{-%
\frac{1}{2}}\right\rangle_K \left\langle w\right\rangle_{K_-} \mathbb{E}%
_{K_-}^{w}\left(\psi w^{-\frac{1}{2}}\right) \right\vert \\
& = & \left\vert \sum_{K\in\mathcal{D}} \left\langle \phi\right\rangle_K 
\frac{\widehat{w}(K_-)}{\left\langle w\right\rangle_{K_-}} \frac{\widehat{%
w^{-1}}(K)}{\left\langle w^{-1}\right\rangle_{K}} \left\langle w^{-\frac{1}{2%
}}\right\rangle_K \left\langle w\right\rangle_{K_-} \mathbb{E}%
_{K_-}^{w}\left(\psi w^{-\frac{1}{2}}\right)\right\vert \\
& = & \left\vert \sum_{K\in\mathcal{D}} \left\langle \phi\right\rangle_K 
\frac{\widehat{w^{-1}}(K) \widehat{w}(K_-)}{\left\langle
w^{-1}\right\rangle_{K}} \left\langle w^{-\frac{1}{2}}\right\rangle_K 
\mathbb{E}_{K_-}^{w}\left(\psi w^{-\frac{1}{2}}\right)\right\vert \\
& \leq & \left(\sum_{K\in\mathcal{D}} \left\langle \phi\right\rangle_K^2 
\frac{\widehat{w^{-1}}(K) \widehat{w}(K_-)}{\left\langle
w^{-1}\right\rangle_{K}} \left\langle w^{-\frac{1}{2}}\right\rangle_K^2
\right)^{\frac{1}{2}}\left(\sum_{K\in\mathcal{D}} \frac{\widehat{w^{-1}}(K) 
\widehat{w}(K_-)}{\left\langle w^{-1}\right\rangle_{K}} \mathbb{E}
_{K_-}^{w}\left(\psi w^{-\frac{1}{2}}\right)^2 \right)^{\frac{1}{2}} \\
& \leq & \left(\sum_{K\in\mathcal{D}} \left\langle \phi\right\rangle_K^2 
\widehat{w^{-1}}(K) \widehat{w}(K_-)\right)^{\frac{1}{2}}\left(\sum_{K\in%
\mathcal{D}} \frac{\widehat{w^{-1}}(K) \widehat{w}(K_-)}{\left\langle
w^{-1}\right\rangle_{K}} \mathbb{E}_{K_-}^{w}\left(\psi w^{-\frac{1}{2}%
}\right)^2 \right)^{\frac{1}{2}}. \\
\end{eqnarray*}
Now note that in \cite{Petermichl} the following estimates are proved 
\begin{eqnarray*}
\sum_{K\subset J} \widehat{w^{-1}}(K) \widehat{w}(K_-) & \lesssim & \left[w%
\right]_{A_2}\left\vert J\right\vert \quad\forall J\in\mathcal{D} \\
\sum_{K\subset J} \frac{\widehat{w^{-1}}(K) \widehat{w}(K_-)}{\left\langle
w^{-1}\right\rangle_{K}} & \lesssim & \left[w\right]_{A_2}w \left( J\right)
\quad\forall J\in\mathcal{D}. \\
\end{eqnarray*}
The Carleson Embedding Theorem applied to $\phi$ and $\psi$ then gives that 
\begin{equation*}
\left\vert S_4\right\vert\lesssim \left[w\right]_{A_2}
\left\Vert\psi\right\Vert_{L^2}\left\Vert \phi\right\Vert_{L^2}
\end{equation*}
as desired.

\subsubsection{Estimating Terms \eqref{Difficult2} and \eqref{Difficult4}}

\label{subsec:Difficult24}

Once again there is symmetry between these terms, so we focus only on %
\eqref{Difficult2}.

Fix $\phi ,\psi \in L^{2}$. We compute 
\begin{eqnarray*}
\left\langle \mathsf{P}_{\widehat{w^{\frac{1}{2}}}}^{(0,1)}\mathcal{S}%
\mathsf{P}_{\left\langle w^{-\frac{1}{2}}\right\rangle }^{(0,0)}\phi ,\psi
\right\rangle _{L^{2}} &=&\left\langle \mathcal{S}\mathsf{P}_{\left\langle
w^{-\frac{1}{2}}\right\rangle }^{(0,0)}\phi ,\mathsf{P}_{\widehat{w^{\frac{1%
}{2}}}}^{(1,0)}\psi \right\rangle _{L^{2}} \\
&=&\sum_{J,K\in \mathcal{D}}\widehat{\phi }(J)\left\langle w^{-\frac{1}{2}%
}\right\rangle _{J}\widehat{\psi }(K)\widehat{w^{\frac{1}{2}}}%
(K)\left\langle h_{K}^{1},h_{J_{-}}\right\rangle _{L^{2}} \\
&=&\sum_{J\in \mathcal{D}}\widehat{\phi }(J)\left\langle w^{-\frac{1}{2}%
}\right\rangle _{J}\left( \frac{1}{\left\vert J_{-}\right\vert ^{\frac{1}{2}}%
}\sum_{K\subsetneq J_{-}}\widehat{\psi }(K)\widehat{w^{\frac{1}{2}}}%
(K)\delta (J_{-},K)\right) \\
&=&\sum_{J\in \mathcal{D}}\widehat{\phi }(J)\left\langle w^{-\frac{1}{2}%
}\right\rangle _{J}\left( \widehat{\psi w^{\frac{1}{2}}}(J_{-})-\widehat{%
\psi }(J_{-})\left\langle w^{\frac{1}{2}}\right\rangle _{J_{-}}-\widehat{w^{%
\frac{1}{2}}}(J_{-})\left\langle \psi \right\rangle _{J_{-}}\right) \\
&\equiv &T_{1}+T_{2}+T_{3}.
\end{eqnarray*}%
We show that each term satisfies $\left\vert T_{j}\right\vert \lesssim \left[
w\right] _{A_{2}}\left\Vert \phi \right\Vert _{L^{2}}\left\Vert \psi
\right\Vert _{L^{2}}$, which would imply that 
\begin{equation*}
\left\Vert \mathsf{P}_{\widehat{w^{\frac{1}{2}}}}^{(0,1)}\mathcal{S}\mathsf{P%
}_{\left\langle w^{-\frac{1}{2}}\right\rangle }^{(0,0)}\right\Vert
_{L^{2}\rightarrow L^{2}}=\sup_{\phi ,\psi \in L^{2}}\left\vert \left\langle 
\mathsf{P}_{\widehat{w^{\frac{1}{2}}}}^{(0,1)}\mathcal{S}\mathsf{P}%
_{\left\langle w^{-\frac{1}{2}}\right\rangle }^{(0,0)}\phi ,\psi
\right\rangle _{L^{2}}\right\vert \lesssim \left[ w\right] _{A_{2}}.
\end{equation*}%
proving the desired norm on the operator in question.

The term $T_2$ is easiest since we have 
\begin{eqnarray*}
\left\vert T_2\right\vert & = & \left\vert \sum_{J\in\mathcal{D}} \widehat{%
\phi}(J)\left\langle w^{-\frac{1}{2}}\right\rangle_{J} \widehat{\psi}%
(J_-)\left\langle w^{\frac{1}{2}}\right\rangle_{J_-}\right\vert \\
& \lesssim & \left[ w\right]_{A_2}^{\frac{1}{2}} \sum_{J\in\mathcal{D}}
\left\vert\widehat{\phi}(J) \widehat{\psi}(J_-)\right\vert \\
& \leq & \left[ w\right]_{A_2}^{\frac{1}{2}}\left\Vert \phi\right\Vert_{L^2}
\left\Vert \psi\right\Vert_{L^2}
\end{eqnarray*}
using the definition of $A_2$, Cauchy-Schwarz and Parseval. For term $T_3$,
we find 
\begin{eqnarray*}
\left\vert T_3\right\vert & = & \left\vert \sum_{J\in\mathcal{D}} \widehat{%
\phi}(J)\left\langle w^{-\frac{1}{2}}\right\rangle_{J} \widehat{w^{\frac{1}{2%
}}}(J_-)\left\langle \psi\right\rangle_{J_-}\right\vert \\
& \leq & \left\Vert \phi\right\Vert_{L^2}\left( \sum_{J\in\mathcal{D}}
\left\langle w^{-\frac{1}{2}}\right\rangle_{J}^2 \widehat{w^{\frac{1}{2}}}%
(J_-)^2\left\langle \psi\right\rangle_{J_-}^2\right)^{\frac{1}{2}} \\
& \lesssim & \left[w \right]_{A_2} \left\Vert
\phi\right\Vert_{L^2}\left\Vert \psi\right\Vert_{L^2}
\end{eqnarray*}
with the last estimate following from the Carleson Embedding Theorem and %
\eqref{ModSquareFunctionEst} (which shows that the Carleson Embedding
Theorem applies).

For term $T_1$, we require disbalanced Haar functions again. Note that it is
possible to write 
\begin{eqnarray*}
T_1 & \equiv & \sum_{J\in\mathcal{D}} \widehat{\phi}(J)\left\langle w^{-%
\frac{1}{2}}\right\rangle_{J} \widehat{\psi w^{\frac{1}{2}}}(J_-) \\
& = & \sum_{J\in\mathcal{D}} \widehat{\phi}(J)\left\langle w^{-\frac{1}{2}%
}\right\rangle_{J} \left\langle\psi w^{\frac{1}{2}}, C_{J_-}(w)
h_{J_-}^w+D_{J_-}(w) h_{J_-}^1\right\rangle_{L^2} \\
& = & \sum_{J\in\mathcal{D}} \widehat{\phi}(J)\left\langle w^{-\frac{1}{2}}
\right\rangle_{J} C_{J_-}(w) \left\langle\psi w^{\frac{1}{2}},
h_{J_-}^w\right\rangle_{L^2} +\sum_{J\in\mathcal{D}} \widehat{\phi}%
(J)\left\langle w^{-\frac{1}{2}}\right\rangle_{J} D_{J_-}(w)
\left\langle\psi w^{\frac{1}{2}}, h_{J_-}^1\right\rangle_{L^2} \\
& = & \sum_{J\in\mathcal{D}} \widehat{\phi}(J)\left\langle w^{-\frac{1}{2}}
\right\rangle_{J} C_{J_-}(w) \left\langle\psi w^{-\frac{1}{2}},
h_{J_-}^w\right\rangle_{L^2(w)} \\
& & + \sum_{J\in\mathcal{D}} \widehat{\phi}(J)\left\langle w^{-\frac{1}{2}%
}\right\rangle_{J} D_{J_-}(w) \left\langle w\right\rangle_{J_-} \mathbb{E}%
_{J_-}^w\left(\psi w^{-\frac{1}{2}}\right) \\
& \equiv & S_1+S_2.
\end{eqnarray*}
For term $S_1$ observe that 
\begin{equation*}
\left\langle w^{-\frac{1}{2}} \right\rangle_{J} C_{J_-}(w)\lesssim
\left\langle w^{-\frac{1}{2}} \right\rangle_{J} \sqrt{\left\langle w
\right\rangle_{J}}\leq \sqrt{\left\langle w^{-1} \right\rangle_{J}
\left\langle w \right\rangle_{J}}\leq\left[ w\right]_{A_2}^{\frac{1}{2}},
\end{equation*}
and so 
\begin{equation*}
\left\vert S_1\right\vert \lesssim \left[ w\right]_{A_2}^{\frac{1}{2}%
}\left\Vert \phi\right\Vert_{L^2} \left\Vert \psi w^{-\frac{1}{2}%
}\right\Vert_{L^2(w)}= \left[ w\right]_{A_2}^{\frac{1}{2}}\left\Vert
\phi\right\Vert_{L^2} \left\Vert \psi\right\Vert_{L^2}.
\end{equation*}
For term $S_2$ we have 
\begin{eqnarray*}
\left\vert S_2\right\vert & = & \left\vert \sum_{J\in\mathcal{D}} \widehat{%
\phi}(J)\left\langle w^{-\frac{1}{2}}\right\rangle_{J} D_{J_-}(w)
\left\langle w\right\rangle_{J_-} \mathbb{E}_{J_-}^w\left(\psi w^{-\frac{1}{2%
}}\right)\right\vert \\
& \leq & \left\Vert \phi\right\Vert_{L^2} \left(\sum_{J\in\mathcal{D}}
\left\langle w^{-\frac{1}{2}}\right\rangle_{J}^2 D_{J_-}(w)^2 \left\langle
w\right\rangle_{J_-}^2 \mathbb{E}_{J_-}^w\left(\psi w^{-\frac{1}{2}%
}\right)^2 \right)^{\frac{1}{2}} \\
& = & \left\Vert \phi\right\Vert_{L^2} \left(\sum_{J\in\mathcal{D}}
\left\langle w^{-\frac{1}{2}}\right\rangle_{J}^2 \frac{\widehat{w}(J_-)^2}{%
\left\langle w\right\rangle_{J_-}^2} \left\langle w\right\rangle_{J_-}^2 
\mathbb{E}_{J_-}^w\left(\psi w^{-\frac{1}{2}}\right)^2 \right)^{\frac{1}{2}}
\\
& = & \left\Vert \phi\right\Vert_{L^2} \left(\sum_{J\in\mathcal{D}}
\left\langle w^{-\frac{1}{2}}\right\rangle_{J}^2 \widehat{w}(J_-)^2\mathbb{E}%
_{J_-}^w\left(\psi w^{-\frac{1}{2}}\right)^2 \right)^{\frac{1}{2}} \\
& \lesssim & \left[ w\right]_{A_2}\left\Vert \phi\right\Vert_{L^2}
\left\Vert \psi\right\Vert_{L^2}
\end{eqnarray*}
with the last estimate following from the Carleson Embedding Theorem and the
linear bound for the modified square function \eqref{ModSquareFunctionEst}.

\subsubsection{Estimating Term \eqref{VeryDifficult}}

\label{subsec:VeryDifficult}

We are left with analyzing $\mathsf{P}_{\widehat{w^{\frac{1}{2}}}}^{(0,1)}\mathcal{S}%
\mathsf{P}_{\widehat{w^{-\frac{1}{2}}}}^{(1,0)}$ as an operator on $L^2$.  But, by the analysis above we have from \eqref{Prod_Expan} that
$$
\mathsf{P}_{\widehat{w^{\frac{1}{2}}}}^{(0,1)}\mathcal{S}%
\mathsf{P}_{\widehat{w^{-\frac{1}{2}}}}^{(1,0)}= M_{w^{\frac{1}{2}}}\mathcal{S} M_{w^{-\frac{1}{2}}}-\eqref{Easy1}-\eqref{Easy2}-\eqref{Easy3}-\eqref{Easy4}-\eqref{Difficult1}-\eqref{Difficult2}-\eqref{Difficult3}-\eqref{Difficult4}.
$$
And, using all the estimates above, we find that
$$
\left\Vert \mathsf{P}_{\widehat{w^{\frac{1}{2}}}}^{(0,1)}\mathcal{S}%
\mathsf{P}_{\widehat{w^{-\frac{1}{2}}}}^{(1,0)}\right\Vert_{L^2\to L^2} \lesssim \left[ w\right]_{A_2}+\left\Vert M_{w^{\frac{1}{2}}}\mathcal{S} M_{w^{-\frac{1}{2}}}\right\Vert_{L^2\to L^2}.
$$
Now using Petermichl's result, Theorem \ref{linearA_2}, (see also \cites{H, H2, MR2657437}) we have that $\left\Vert M_{w^{\frac{1}{2}}}\mathcal{S} M_{w^{-\frac{1}{2}}}\right\Vert_{L^2\to L^2}\lesssim \left[w\right]_{A_2}$.  This combines to give:
$$
\left\Vert \mathsf{P}_{\widehat{w^{\frac{1}{2}}}}^{(0,1)}\mathcal{S}%
\mathsf{P}_{\widehat{w^{-\frac{1}{2}}}}^{(1,0)}\right\Vert_{L^2\to L^2} \lesssim \left[ w\right]_{A_2}
$$
as required.

We pose a question that we are unable to resolve as of this writing.  Resolution of this question would provide an alternate proof of the linear bound of the Hilbert transform.

\begin{question}
\label{question1}
Can one give a direct proof that
$$
\left\Vert \mathsf{P}_{\widehat{w^{\frac{1}{2}}}}^{(0,1)}\mathcal{S}%
\mathsf{P}_{\widehat{w^{-\frac{1}{2}}}}^{(1,0)}\right\Vert_{L^2\to L^2} \lesssim \left[ w\right]_{A_2}
$$
without resorting to the use of Theorem \ref{linearA_2}?  A direct proof of the above estimate would provide an alternate proof of Theorem \ref{linearA_2}.
\end{question}

\subsection{An Alternate Approach to Estimating Term \eqref{VeryDifficult}}

In this final section we show that the tools of this paper are almost sufficient to provide a direct proof of \eqref{VeryDifficult}.

Fix $\phi,\psi\in L^2$. Now, observe that we have 
\begin{eqnarray}
\left\langle \mathsf{P}_{\widehat{w^{\frac{1}{2}}}}^{(0,1)}\mathcal{S}%
\mathsf{P}_{\widehat{w^{-\frac{1}{2}}}}^{(1,0)}\phi ,\psi\right\rangle
_{L^{2} } &=&\left\langle \mathcal{S}\mathsf{P}_{\widehat{w^{-\frac{1}{2}}}%
}^{(1,0)}\phi,\mathsf{P}_{\widehat{w^{\frac{1}{2}}}}^{(1,0)}\psi\right%
\rangle _{L^{2} }  \notag \\
&=&\sum_{J,L\in \mathcal{D}}\widehat{w^{-\frac{1}{2}}}(J)\widehat{w^{\frac{1%
}{2}}}(L)\widehat{\phi}(J)\widehat{\psi}(L)\left\langle \mathcal{S}
h_{J}^{1},h_{L}^{1}\right\rangle _{L^{2} } .  \notag
\end{eqnarray}
We then further specialize to the case when $J\cap L=\emptyset$ and when $J\cap L\neq\emptyset$.  The case when $J\cap L\neq \emptyset$ can be handled via the techniques of this paper.  The case when $J\cap L=\emptyset$ requires a new idea.

\subsubsection{Estimating Term \eqref{VeryDifficult}:  $J\cap L\neq \emptyset$}

We can then split this into three sums, when $L=J$, when $L\subsetneq J$ and
when $J\subsetneq L$. Doing so, we see that

\begin{eqnarray*}
\left\langle \mathsf{P}_{\widehat{w^{\frac{1}{2}}}}^{(0,1)}\mathcal{S}%
\mathsf{P}_{\widehat{w^{-\frac{1}{2}}}}^{(1,0)}\phi ,\psi\right\rangle
_{L^{2} } & = & \sum_{J\in\mathcal{D}} \widehat{w^{-\frac{1}{2}}}(J)\widehat{%
w^{\frac{1}{2}}}(J)\widehat{\phi}(J) \widehat{\psi}(J)\left\langle \mathcal{S%
} h_{J}^{1},h_{J}^{1}\right\rangle _{L^{2} } \\
& & +\sum_{J\in \mathcal{D}}\widehat{w^{-\frac{1}{2}}}(J)\widehat{\phi}%
(J)\left(\sum_{L\subsetneq J} \widehat{w^{\frac{1}{2}}}(L) \widehat{\psi}%
(L)\left\langle \mathcal{S} h_{J}^{1},h_{L}^{1}\right\rangle _{L^{2} }
\right) \\
& & + \sum_{L\in \mathcal{D}} \widehat{w^{\frac{1}{2}}}(L) \widehat{\psi}(L)
\left(\sum_{J\subsetneq L} \widehat{w^{-\frac{1}{2}}}(J) \widehat{\phi}(J)
\left\langle \mathcal{S} h_{J}^{1},h_{L}^{1}\right\rangle _{L^{2} }\right) \\
& = & T_1+T_2+T_3.
\end{eqnarray*}
Consider now term $T_1$ and observe 
\begin{eqnarray*}
\left\vert T_1\right\vert & = & \left\vert \sum_{J\in\mathcal{D}} \widehat{%
w^{-\frac{1}{2}}}(J)\widehat{w^{\frac{1}{2}}}(J)\widehat{\phi}(J) \widehat{%
\psi}(J)\left\langle \mathcal{S} h_{J}^{1},h_{J}^{1}\right\rangle _{L^{2}
}\right\vert \\
& \leq & \sum_{J\in\mathcal{D}} \left\vert \widehat{w^{-\frac{1}{2}}}(J)%
\widehat{w^{\frac{1}{2}}}(J)\widehat{\phi}(J) \widehat{\psi}(J)\right\vert
\left\Vert h_J^1 \right\Vert_{L^2}^2 \\
& \leq & \sum_{J\in\mathcal{D}} \left\vert\frac{\widehat{w^{-\frac{1}{2}}}(J)%
\widehat{w^{\frac{1}{2}}}(J)}{\left\vert J\right\vert}\widehat{\phi}(J) 
\widehat{\psi}(J)\right\vert \\
& \leq & \left[ w\right]_{A_2}^{\frac{1}{2}}\left\Vert
\phi\right\Vert_{L^2}\left\Vert \psi\right\Vert_{L^2}.
\end{eqnarray*}
Here we make an obvious estimate using the definition of $A_2$,
Cauchy-Schwarz and that $\mathcal{S}:L^2\rightarrow L^2$ has norm at most
one.

There is a symmetry between terms $T_{2}$ and $T_{3}$, so we only handle
term $T_{2}$. We first make a computation of $\left\langle \mathcal{S}%
h_{J}^{1},h_{L}^{1}\right\rangle _{L^{2}}$ when $L\subsetneq J$. Set 
\begin{equation*}
\mathfrak{s}(J)\equiv \sqrt{2}\sum_{K\in \mathcal{D}:K_{-}\supsetneq J}\frac{%
1}{\left\vert K\right\vert }.
\end{equation*}%
Letting $h_{K}(J)\equiv h_{K}(c(J))$ wherein $c(J)$ is the center of the
dyadic interval $J$ and using the definition of $\mathcal{S}$ it is a
straightforward computation to show that 
\begin{eqnarray*}
\left\langle \mathcal{S}h_{J}^{1},h_{L}^{1}\right\rangle _{L^{2}}
&=&\sum_{K\supsetneq J}h_{K}(J)\left\langle h_{L}^{1},h_{K_{-}}\right\rangle
_{L^{2}} \\
&=&\sum_{\substack{ K_{-}\supsetneq L  \\ K\supsetneq J}}h_{K}(J)h_{K_{-}}(L)
\\
&=&\sum_{\substack{ K_{-}\supsetneq L  \\ K\supsetneq J}}\frac{\delta (J,K)}{%
\left\vert K\right\vert ^{\frac{1}{2}}}\frac{\delta (L,K_{-})}{\left\vert
K_{-}\right\vert ^{\frac{1}{2}}} \\
&=&\sqrt{2}\sum_{\substack{ K_{-}\supsetneq L  \\ K\supsetneq J}}\frac{%
\delta (J,K)\delta (L,K_{-})}{\left\vert K\right\vert } \\
&=&\sqrt{2}\sum_{K\in \mathcal{D}:K_{-}\supsetneq J}\frac{1}{\left\vert
K\right\vert }\equiv \mathfrak{s}(J).
\end{eqnarray*}%
We have used that $\left\vert K_{-}\right\vert =\frac{1}{2}\left\vert
K\right\vert $, and that $\delta (J,K)=\delta (L,K_{-})$ since $L\subsetneq
J\subsetneq K$. Note that we have 
\begin{equation*}
\mathfrak{s}(J)\leq \sqrt{2}\sum_{K\supsetneq J}\frac{1}{\left\vert
K\right\vert }=\frac{\sqrt{2}}{\left\vert J\right\vert }.
\end{equation*}%
Now define the function $v_{J}\equiv w^{\frac{1}{2}}\left( \mathbf{1}%
_{J_{-}}-\mathbf{1}_{J_{+}}\right) $ and observe 
\begin{equation*}
\widehat{v_{J}}(L)=\left\{ 
\begin{array}{ccc}
\widehat{w^{\frac{1}{2}}}(L) & : & L\subset J_{-} \\ 
-\widehat{w^{\frac{1}{2}}}(L) & : & L\subset J_{+}.%
\end{array}%
\right.
\end{equation*}%
Then using the product formula we have 
\begin{eqnarray*}
\sum_{L\subsetneq J}\widehat{w^{\frac{1}{2}}}(L)\widehat{\psi }%
(L)\left\langle \mathcal{S}h_{J}^{1},h_{L}^{1}\right\rangle _{L^{2}} &=&%
\mathfrak{s}(J)\sum_{L\subsetneq J}\widehat{w^{\frac{1}{2}}}(L)\widehat{\psi 
}(L) \\
&=&\mathfrak{s}(J)\left( \sum_{L\subset J_{-}}\widehat{w^{\frac{1}{2}}}(L)%
\widehat{\psi }(L)-\sum_{L\subset J_{+}}(-\widehat{w^{\frac{1}{2}}}(L))%
\widehat{\psi }(L)\right) \\
&=&\mathfrak{s}(J)\left\vert J\right\vert ^{\frac{1}{2}}\frac{1}{\left\vert
J\right\vert ^{\frac{1}{2}}}\sum_{L\subsetneq J}\widehat{\psi }(L)\widehat{%
v_{J}}(L)\delta (L,J) \\
&=&\mathfrak{s}(J)\left\vert J\right\vert ^{\frac{1}{2}}\left( \widehat{%
v_{J}\psi }(J)-\widehat{\psi }(J)\left\langle v_{J}\right\rangle _{J}-%
\widehat{v_{J}}(J)\left\langle \psi \right\rangle _{J}\right) .
\end{eqnarray*}%
However, simple computations show that 
\begin{eqnarray*}
\left\langle v_{J}\right\rangle _{J} &=&\left\vert J\right\vert ^{-\frac{1}{2%
}}\widehat{w^{\frac{1}{2}}}(J) \\
\widehat{v_{J}}(J) &=&\left\vert J\right\vert ^{\frac{1}{2}}\left\langle w^{%
\frac{1}{2}}\right\rangle _{J} \\
\widehat{\psi v_{J}}(J) &=&\left\vert J\right\vert ^{\frac{1}{2}%
}\left\langle \psi w^{\frac{1}{2}}\right\rangle _{J}.
\end{eqnarray*}%
Thus, we have 
\begin{equation*}
\sum_{L\subsetneq J}\widehat{w^{\frac{1}{2}}}(L)\widehat{\psi }%
(L)\left\langle \mathcal{S}h_{J}^{1},h_{L}^{1}\right\rangle _{L^{2}}=%
\mathfrak{s}(J)\left\vert J\right\vert \left\langle \psi w^{\frac{1}{2}%
}\right\rangle _{J}-\mathfrak{s}(J)\widehat{w^{\frac{1}{2}}}(J)\widehat{\psi 
}(J)-\mathfrak{s}(J)\left\vert J\right\vert \left\langle w^{\frac{1}{2}%
}\right\rangle _{J}\left\langle \psi \right\rangle _{J}.
\end{equation*}%
This then yields that 
\begin{eqnarray*}
T_{2} &=&\sum_{J\in \mathcal{D}}\widehat{w^{-\frac{1}{2}}}(J)\widehat{\phi }%
(J)\left( \sum_{L\subsetneq J}\widehat{w^{\frac{1}{2}}}(L)\widehat{\psi }%
(L)\left\langle \mathcal{S}h_{J}^{1},h_{L}^{1}\right\rangle _{L^{2}}\right)
\\
&=&\sum_{J\in \mathcal{D}}\widehat{w^{-\frac{1}{2}}}(J)\widehat{\phi }%
(J)\left( \mathfrak{s}(J)\left\vert J\right\vert \left\langle \psi w^{\frac{1%
}{2}}\right\rangle _{J}-\mathfrak{s}(J)\widehat{w^{\frac{1}{2}}}(J)\widehat{%
\psi }(J)-\mathfrak{s}(J)\left\vert J\right\vert \left\langle w^{\frac{1}{2}%
}\right\rangle _{J}\left\langle \psi \right\rangle _{J}\right) \\
&\equiv &S_{1}+S_{2}+S_{3}.
\end{eqnarray*}%
Using the estimate that $\mathfrak{s}(J)\lesssim \left\vert J\right\vert
^{-1}$, it is immediate to see that 
\begin{equation*}
\left\vert S_{2}\right\vert \leq \left[ w\right] _{A_{2}}^{\frac{1}{2}%
}\left\Vert \phi \right\Vert _{L^{2}}\left\Vert \psi \right\Vert _{L^{2}}.
\end{equation*}%
It is also an easy application of the Carleson Embedding Theorem, again
using the linear bound for the square function and that $\mathfrak{s}%
(J)\left\vert J\right\vert \lesssim 1$, to see that 
\begin{equation*}
\left\vert S_{3}\right\vert \lesssim \left\Vert \phi \right\Vert
_{L^{2}}\left( \sum_{J\in \mathcal{D}}\widehat{w^{-\frac{1}{2}}}%
(J)^{2}\left\langle w^{\frac{1}{2}}\right\rangle _{J}^{2}\left\langle \psi
\right\rangle _{J}^{2}\right) ^{\frac{1}{2}}\lesssim \left[ w\right]
_{A_{2}}\left\Vert \phi \right\Vert _{L^{2}}\left\Vert \psi \right\Vert
_{L^{2}}.
\end{equation*}%
Finally, note that for term $S_{1}$, by an application of Cauchy-Schwarz and
again that $\mathfrak{s}(J)\left\vert J\right\vert \lesssim 1$, we have 
\begin{eqnarray*}
\left\vert S_{1}\right\vert &=&\left\vert \sum_{J\in \mathcal{D}}\mathfrak{s}%
(J)\left\vert J\right\vert \widehat{w^{-\frac{1}{2}}}(J)\widehat{\phi }%
(J)\left\langle \psi w^{\frac{1}{2}}\right\rangle _{J}\right\vert \\
&=&\left\vert \sum_{J\in \mathcal{D}}\mathfrak{s}(J)\left\vert J\right\vert 
\widehat{w^{-\frac{1}{2}}}(J)\widehat{\phi }(J)\left\langle w\right\rangle
_{J}\mathbb{E}_{J}^{w}\left( \psi w^{-\frac{1}{2}}\right) \right\vert \\
&\lesssim &\left\Vert \phi \right\Vert _{L^{2}}\left( \sum_{J\in \mathcal{D}}%
\widehat{w^{-\frac{1}{2}}}(J)^{2}\left\langle w\right\rangle _{J}^{2}\mathbb{%
E}_{J}^{w}\left( \psi w^{-\frac{1}{2}}\right) ^{2}\right) ^{\frac{1}{2}}.
\end{eqnarray*}%
We claim that the following Carleson estimate holds:

\begin{equation}
\sum_{K\subset L}\widehat{w^{-\frac{1}{2}}}\left( K\right) ^{2}\left\langle
w\right\rangle _{K}^{2}\lesssim \left[ w\right] _{A_{2}}^{2}w(L)\quad
\forall L\in \mathcal{D}.  \label{if we can1}
\end{equation}%
Assuming \eqref{if we can1} holds, we can apply the Carleson Embedding
Theorem to conclude that 
\begin{equation*}
\left\vert S_{1}\right\vert \lesssim \left[ w\right] _{A_{2}}\left\Vert \phi
\right\Vert _{L^{2}}\left\Vert \psi w^{-\frac{1}{2}}\right\Vert _{L^{2}(w)}=%
\left[ w\right] _{A_{2}}\left\Vert \phi \right\Vert _{L^{2}}\left\Vert \psi
\right\Vert _{L^{2}}.
\end{equation*}%
Then combining the estimates on $S_{j}$ with $j=1,2,3$ gives that 
\begin{equation*}
\left\vert T_{2}\right\vert \lesssim \left[ w\right] _{A_{2}}\left\Vert \phi
\right\Vert _{L^{2}}\left\Vert \psi \right\Vert _{L^{2}}.
\end{equation*}%
Further, combining the estimates on $T_{j}$ for $j=1,2,3$ gives that 
\begin{equation*}
\sup_{\phi ,\psi \in L^{2}}\left\vert \left\langle \mathsf{P}_{\widehat{w^{%
\frac{1}{2}}}}^{(0,1)}\mathcal{S}\mathsf{P}_{\widehat{w^{-\frac{1}{2}}}%
}^{(1,0)}\phi ,\psi \right\rangle _{L^{2}}\right\vert \lesssim \left[ w%
\right] _{A_{2}}\left\Vert \phi \right\Vert _{L^{2}}\left\Vert \psi
\right\Vert _{L^{2}},
\end{equation*}%
which is the desired estimate.

Thus it only remains to demonstrate \eqref{if we can1}. Fix $L$ and start
with a standard Calder\'{o}n--Zygmund stopping time argument on the function 
$w$. Let $Q_{1}^{0}\equiv L$. Fix $\gamma >1$ large and let $\mathcal{G}%
_{0}=\left\{ Q_{1}^{0}\right\} $ be the ground zero generation. Then let $%
\mathcal{G}_{1}=\mathcal{G}_{1}\left( Q_{1}^{0}\right) =\left\{
Q_{j}^{1}\right\} _{j\in \Gamma _{1}^{0}}$ be the first generation of
maximal dyadic subintervals $Q_{j}^{1}$ of $L$ such that%
\begin{equation*}
\left\langle w\right\rangle _{Q_{j}^{1}}>\gamma \left\langle w\right\rangle
_{Q_{1}^{0}},\ \ \ \ \ j\in \Gamma _{1}^{0}.
\end{equation*}%
Then for each $i\in \Gamma _{1}^{0}$ define $\mathcal{G}_{2}\left(
Q_{i}^{1}\right) =\left\{ Q_{j}^{2}\right\} _{j\in \Gamma _{i}^{1}}$ to be
the collection of maximal dyadic subintervals $Q_{j}^{2}$ of $Q_{i}^{1}$
such that%
\begin{equation*}
\left\langle w\right\rangle _{Q_{j}^{2}}>\gamma \left\langle w\right\rangle
_{Q_{i}^{1}},\ \ \ \ \ j\in \Gamma _{i}^{1}.
\end{equation*}%
Set $\mathcal{G}_{2}=\bigcup\limits_{Q_{i}^{1}\in \mathcal{G}_{1}}\mathcal{G}%
_{2}\left( Q_{i}^{1}\right) $ to be the second generation of subintervals of 
$L$. Continue by recursion to define the $k^{th}$ generation $\mathcal{G}%
_{k} $ of subintervals of $L$ by $\mathcal{G}_{k}\equiv
\bigcup\limits_{Q_{i}^{k-1}\in \mathcal{G}_{k-1}}\mathcal{G}_{k}\left(
Q_{i}^{k-1}\right) $ where $\mathcal{G}_{k}\left( Q_{i}^{k-1}\right)
=\left\{ Q_{j}^{k}\right\} _{j}$ is the collection of maximal dyadic
subintervals $Q_{j}^{k}$ of $Q_{i}^{k-1}$ such that%
\begin{equation*}
\left\langle w\right\rangle _{Q_{j}^{k}}>\gamma \left\langle w\right\rangle
_{Q_{i}^{k-1}},\ \ \ \ \ j\in \Gamma _{i}^{k-1}.
\end{equation*}%
Then define the corona 
\begin{equation*}
\mathcal{C}_{Q_{j}^{k}}=\left\{ K\subset Q_{j}^{k}:K\not\subset Q_{i}^{k+1}%
\text{ for any }Q_{i}^{k+1}\in \mathcal{G}_{k+1}\right\} ,
\end{equation*}%
and denote by $\mathcal{G}=\bigcup\limits_{k=0}^{\infty }\mathcal{G}_{k}$
the collection of all stopping intervals in $L$. Within each corona
\thinspace $\mathcal{C}_{G}$ for $G\in \mathcal{G}$, we get%
\begin{eqnarray*}
\sum_{K\in \mathcal{C}_{G}}\widehat{w^{-\frac{1}{2}}}\left( K\right)
^{2}\left\langle w\right\rangle _{K}^{2} &\lesssim &\gamma \left( \sum_{K\in 
\mathcal{C}_{G}}\widehat{w^{-\frac{1}{2}}}\left( K\right) ^{2}\right)
\left\langle w\right\rangle _{G}^{2} \\
&\leq &\gamma \left( \int_{G}\left\vert w^{-\frac{1}{2}}-\left\langle w^{-%
\frac{1}{2}}\right\rangle _{G}\right\vert ^{2}\right) \left\langle
w\right\rangle _{G}^{2} \\
&\leq &\gamma \int_{G}\left\langle w\right\rangle _{G}^{2}w^{-1}.
\end{eqnarray*}%
Summing up over all the coronas $\mathcal{C}_{G}$ for $G\in \mathcal{G}$
with $\gamma =2$ we get%
\begin{eqnarray*}
\sum_{K\subset L}\widehat{w^{-\frac{1}{2}}}\left( K\right) ^{2}\left\langle
w\right\rangle _{K}^{2} &\lesssim &\sum_{G\in \mathcal{G}}\int_{G}\left%
\langle w\right\rangle _{G}^{2}w^{-1} \\
&=&\sum\limits_{k,j}\left\vert Q_{j}^{k}\right\vert _{w^{-1}}\left( \frac{1}{%
\left\vert Q_{j}^{k}\right\vert }\int_{Q_{j}^{k}}w\right) ^{2} \\
&=&\int_{\mathbb{R}}\sum\limits_{k,j}\left( \frac{1}{\left\vert
Q_{j}^{k}\right\vert }\int_{Q_{j}^{k}}w\right) ^{2}\mathbf{1}%
_{Q_{j}^{k}}\left( x\right) w^{-1}\left( x\right) dx \\
&\approx &\int_{\mathbb{R}}\left( \sup_{k.j:\ x\in Q_{j}^{k}}\frac{1}{%
\left\vert Q_{j}^{k}\right\vert }\int_{Q_{j}^{k}}w\right) ^{2}w^{-1}\left(
x\right) dx \\
&\lesssim &\int_{L}\left\vert Mw\right\vert ^{2}w^{-1}\lesssim \left[ w%
\right] _{A_{2}}^{2}\int_{L}w^{2}w^{-1}=\left[ w\right] _{A_{2}}^{2}%
\int_{L}w,
\end{eqnarray*}%
which gives \eqref{if we can1}. Note that we used that the sequence of
numbers $\left\{ \frac{1}{\left\vert Q_{j}^{k}\right\vert }%
\int_{Q_{j}^{k}}w\right\} _{k,j:\ x\in Q_{j}^{k}}$ is super-geometric in the
sense that consecutive terms have ratio exceeding $\gamma =2$, and so the
sum of their squares is essentially the square of the largest one.

\subsubsection{Estimating Term \eqref{VeryDifficult}: $J\cap L=\emptyset$}
We must consider this case since the Haar shift is not a completely local operator.  It instead sees some long range interaction, and we need to account for this since there can be terms for which $L\cap J=\emptyset$ and $\left\langle \mathcal{S}
h_{J}^{1},h_{L}^{1}\right\rangle _{L^{2} }\neq 0$. A simple example occurs
when $L$ and $J$ are dyadic brothers.   Fortunately, in the case of the identity operator, this term does not appear and the proof would terminate.

The tools used in this paper currently appear to be unable to resolve this term.  We thus pose a refined version of Question \ref{question1}

\begin{question}
\label{question2}
Is the following estimate
$$
\left\Vert \sum_{J,L\in \mathcal{D}: J\cap L=\emptyset}\widehat{w^{-\frac{1}{2}}}(J)\widehat{w^{\frac{1%
}{2}}}(L)\widehat{\phi}(J)\widehat{\psi}(L)\left\langle \mathcal{S}
h_{J}^{1},h_{L}^{1}\right\rangle _{L^{2} }\right\Vert_{L^2\to L^2}\lesssim \left[ w\right]_{A_2}
$$
true?
\end{question}
We again know that this estimate is true by using Theorem \ref{linearA_2}, but are interested in a direct resolution of Question \ref{question2}.  A direct proof of the estimate above would provide a new proof of Theorem \ref{linearA_2}.

\end{document}